\documentclass{amsart}

\usepackage{amssymb,amsthm,a4wide}
\usepackage{epsfig,color}
\usepackage[all,2cell]{xy} \UseAllTwocells \SilentMatrices
\usepackage{stmaryrd}
\usepackage{MnSymbol}

\newtheorem{theorem}{Theorem}[section]

\newtheorem{lemma}[theorem]{Lemma}

\newtheorem{prop}[theorem]{Proposition}
\newtheorem{cor}[theorem]{Corollary}

\newtheorem{question}[theorem]{Question}
\theoremstyle{remark}
\newtheorem{remark}[theorem]{Remark}
\newtheorem{example}[theorem]{Example}
\theoremstyle{definition}

\renewcommand\P{\mathbb{P}}
\newcommand\A{\mathbb{A}}
\newcommand\Z{\mathbb{Z}}
\newcommand\F{\mathbb{F}}
\newcommand\HH{\mathbb{H}}
\newcommand\Q{\mathbb{Q}}
\newcommand\R{\mathbb{R}}

\newcommand\ksep{k^{\text{sep}}}
\newcommand\C{\mathcal{C}}

\newcommand\pC{\overline{\mathcal{C}}}

\newcommand\E{\mathcal{E}}
\renewcommand\O{\mathcal{O}}

\newcommand\ovp{\overline{p}}
\newcommand\ovq{\overline{q}}
\newcommand\ovr{\overline{r}}
\newcommand\kbar{\overline{k}}

\newcommand\fa{\mathfrak{a}}
\newcommand\fb{\mathfrak{b}}
\newcommand\fc{\mathfrak{c}}
\newcommand\m{\mathfrak{m}}

\newcommand\Spec{\mathop{\rm Spec} \nolimits}

\newcommand\Pic{\mathop{\rm Pic} \nolimits}
\newcommand\Div{\mathop{\rm Div} \nolimits}
\newcommand\Ig{\mathop{\rm Ig} \nolimits}

\newcommand\ns{{\rm ns}}
\newcommand\ord{{\rm ord}}

\newcommand\id{{\rm id}}

\newcommand\isom{\cong}

\newenvironment{computations}
{\begin{changemargin}{10mm}{5mm}}
{\end{changemargin}}

\title{Density of rational points on del Pezzo surfaces of degree one}
\author{Cec\'ilia Salgado}
\address{Instituto de Matem\'atica, Universidade Federal do Rio de Janeiro, Ilha do Fund\~{a}o, 
21941-909 Rio de Janeiro, Brasil}
\email{salgado@im.ufrj.br}
\urladdr{www.im.ufrj.br/~salgado}

\author{Ronald van Luijk}
\address{Mathematisch Instituut, Universiteit Leiden, Postbus 9512, 2300 RA, Leiden, Netherlands}
\address{Section de math\'ematiques EPFL, FSB-SMA, Station 8 - B\^atiment MA, CH-1015 Lausanne}
\email{rvl@math.leidenuniv.nl}
\urladdr{http://www.math.leidenuniv.nl/\~{}rvl/}

\begin{document}

\begin{abstract}
We state conditions under which the set $S(k)$ of $k$-rational points on a del Pezzo surface $S$ of degree $1$ over an infinite field $k$ of characteristic not equal to $2$ or $3$ is Zariski dense. For example, it suffices to require that the elliptic fibration $S\dashrightarrow \P^1$ induced by the anticanonical map has a nodal fiber over a $k$-rational point of $\P^1$. 
It also suffices to require the existence of a point in $S(k)$ that does not lie on six exceptional curves of $S$ and that has order $3$ on its fiber of the elliptic fibration.  This allows us to show that within a parameter space for del Pezzo surfaces of degree $1$ over $\R$, the set of surfaces $S$ defined over $\Q$ for which the set $S(\Q)$ is Zariski dense, is dense with respect to the real analytic topology. We also include conditions that may be satisfied for every del Pezzo surface $S$ and that can be verified with a finite computation for any del Pezzo surface $S$ that does satisfy them. 
\end{abstract}
\maketitle

\section{Introduction}\label{intro}

A del Pezzo surface 
over a field $k$ is a smooth, projective, geometrically integral surface $S$ over $k$ with ample anticanonical divisor $-K_S$; the degree of $S$ is defined to be the self-intersection number $d = K_S^2 \geq 1$. A del Pezzo surface is minimal if and only if there is no birational morphism over its ground field to a del Pezzo surface of higher degree.  
Every del Pezzo surface of degree $d$ is geometrically isomorphic to $\P^2$ blown up at $9-d$ points in general position, or to $\P^1 \times \P^1$ if $d=8$. Conversely, every smooth, projective surface that is geometrically birationally equivalent to $\P^2$ is birationally equivalent over the ground field to a del Pezzo surface or a conic bundle (see \cite{iskovskikh}).

A surface $S$ over a field $k$ is unirational if there is a dominant rational map $\P^2 \dashrightarrow S$ over $k$. Segre and Manin  proved that every del Pezzo surface $S$ of degree $d\geq 2$ over a field $k$ with a $k$-rational point is unirational, at least if one assumes that the point 
is in general position in the case $d=2$. For references, see \cite{segreone, segretwo} for $d=3$ and $k=\Q$, see \cite[Theorem 29.4 and 30.1]{manicub} for $d=2$ and $d\geq 5$, as well as $d=3,4$ under the assumption that $k$ is large enough, and see \cite[Theorem 1.1]{kollar} and \cite[Proposition 5.19]{pieropan} for $d=3$ and $d=4$ in general.  
On the other hand, even though del Pezzo surfaces of degree $1$ always have a rational point, 
we do not know whether any minimal del Pezzo surface of degree $1$ that is not birationally equivalent to a conic bundle is unirational over its ground field. 
If $k$ is infinite, then unirationality of $S$ implies that the set $S(k)$ of $k$-rational points on $S$ is Zariski dense. 
The following question asks whether this weaker property may hold for all del Pezzo surfaces of degree $1$. 

\begin{question}\label{Q:alwaysdense}
If $S$ is a del Pezzo surface of degree $1$ over an infinite field $k$, is the set $S(k)$ of $k$-rational points Zariski dense in $S$?
\end{question}

Over number fields, a positive answer to this question is implied by the conjecture by Colliot-Th\'el\`ene and Sansuc that the Brauer--Manin obstruction to weak approximation is the only one for geometrically rational varieties \cite[Conjecture d), p.~319]{CT1}. This conjecture may in fact hold more generally for geometrically rationally connected varieties over global fields (see \cite[p.~3]{CT2} for number fields).

The primary goal of this paper is to state conditions under which the answer to Question \ref{Q:alwaysdense} is positive. 

Let $k$ be a field of characteristic not equal to $2$ or $3$, and $S$ a del Pezzo surface of degree $1$ over $k$ with a canonical divisor $K_S$. 
Then the linear system $|-3K_S|$ induces an embedding of $S$ in the weighted projective space $\P(2,3,1,1)$ with coordinates $x,y,z,w$. More precisely, there are homogeneous polynomials $f,g \in k[z,w]$ of degrees $4$ and $6$, respectively, such that $S$ is isomorphic to the smooth sextic in $\P(2,3,1,1)$ given by 
\begin{equation}\label{maineq}
y^2 = x^3 + f(z,w)x + g(z,w).
\end{equation}
For some special families of del Pezzo surfaces of degree $1$ it is known that the set of rational points is Zariski dense. Examples that are minimal and have no conic bundle structure include those given by A.~V\'arilly-Alvarado. He proves in \cite[Theorem 2.1]{varilly}
that if we have $k=\Q$, while $f$ is zero and $g$ satisfies some technical conditions, 
then the set of $\Q$-rational points on the surface $S$ given by \eqref{maineq} is Zariski dense if one also assumes that Tate--Shafarevich groups of elliptic curves over $\Q$ with $j$-invariant $0$ are finite (cf.~Example \ref{ex:tony}). These technical conditions are satisfied if $g = az^6+bw^6$ for nonzero integers $a,b \in \Z$ with $3ab$ not a square, or with $a$ and $b$ relatively prime and $9\nmid ab$ \cite[Theorem 1.1]{varilly}. 

M.~Ulas \cite{ulasone, ulastwo}, as well as M.~Ulas and A.~Togb\'e \cite[Theorem 2.1]{ulasthree}, also give various conditions on the homogeneous polynomials $f,g \in \Q[z,w]$ for the set of rational points on the surface $S \subset \P(2,3,1,1)$ over $\Q$ given by \eqref{maineq} to be Zariski dense. Besides hypotheses that imply that $S$ is not smooth or not minimal, all their conditions imply that either (i) $f=0$ and $g(t,1)$ is monic, or (ii) $g(t,1)$ has degree at most $4$, or (iii) $f=0$ and $g$ vanishes on a rational point of $\P^1$. 
E.\ Jabara generalizes Ulas' work on case (iii) in \cite[Theorems C and D]{jabara} and treats the case over $\Q$ with $g(t,1)$ monic and the pair $(f,g)$ sufficiently general. 

The techniques in this paper are a generalization of a geometric interpretation of Ulas' work on case (iii); they are independent of the work of Jabara (see Remark \ref{genulas}). 
The projection $\varphi \colon \P(2,3,1,1) \dashrightarrow \P^1$ onto the last two coordinates is a morphism on the complement $U$ of the line given by $z=w=0$ in $\P(2,3,1,1)$. 
For any point $Q \in S(k)$ not equal to $\O = (1:1:0:0)$, we let $\C_Q(5)$ denote the family of sections of $U \to \P^1$ that meet $S$ at $Q$ with multiplicity at least $5$; we will see that $\C_Q(5)$ has the structure of an affine curve, the components of which have genus at most $1$ (see paragraph containing \eqref{CQ5}). 

The restriction $\varphi|_S \colon S \dashrightarrow \P^1$ corresponds to the linear system $|-K_S|$ and has a unique base point $\O \in S$. This map $\varphi|_S$ induces an elliptic fibration $\pi \colon \E \to \P^1$ of the blow-up $\E$ of $S$ at $\O$. 
The exceptional curve on $\E$ above $\O$ is a section, also denoted by $\O$. 
For any $t = (z_0:w_0) \in \P^1$, the fiber $\E_t$ is isomorphic to the intersection of $S$ with the plane $H_t$ given by $w_0z = z_0w$; 
the set $\E_t^\ns(k)$ of smooth $k$-points on $\E_t$ naturally carries a group structure characterized by
the property that three points in $H_t \cap S$ sum to the identity $\O$ if and only if they are collinear. 
Our first main result is the following.

\begin{theorem}\label{maincor}
Let $k$ be an infinite field of characteristic not equal to $2$ or $3$. Let $S \subset \P(2,3,1,1)$ be a del Pezzo surface given by \eqref{maineq} with $f,g \in k[z,w]$, and $\pi \colon \E \to \P^1$ the elliptic fibration induced by the anticanonical map $S \dashrightarrow \P^1$. Let $Q \in S(k)$ be a point that is not fixed by the automorphism of~$S$ that changes the sign of $y$. Let $\C_Q(5)$ be the curve of those sections of the projection $U\to \P^1$ that meet $S$ at the point $Q$ with multiplicity at least $5$. Set $t = \pi(Q)$. Suppose that the following statements hold.
\begin{itemize}
\item The order of $Q$ in $\E_t^\ns(k)$ is at least $3$. 
\item If the order of $Q$ in $\E_t^\ns(k)$ is at least $4$, then $\C_Q(5)$ has infinitely many $k$-points. 
\item If the characteristic of $k$ equals $5$, then the order of $Q$ in $\E_t^\ns(k)$ is not $5$.
\item If the order of $Q$ in $\E_t^\ns(k)$ is $3$ or $5$, then $Q$ does not lie on six $(-1)$-curves of $S$. 
\end{itemize}
Then the set $S(k)$ of $k$-points on $S$ is Zariski dense in $S$. 
\end{theorem}

Note that all four assumptions of Theorem \ref{maincor} are hypotheses on the point $Q$.
Given $S$, we provide an explicit zero-dimensional scheme of which the points correspond to the $(-1)$-curves of~$S$ going through $Q$ (cf.~Remark \ref{minonecurves}), so the first and the last two conditions of Theorem \ref{maincor} are easy to check.
If the set $S(k)$ is indeed Zariski dense in $S$, then 
the subset of those points $Q \in S(k)$ that satisfy these three 
conditions is also dense; Theorem \ref{maincor} provides a proof 
of Zariski density of $S(k)$ as soon as $\C_Q(5)(k)$ is infinite for one of these points $Q$. 
If the answer to Question \ref{Q:alwaysdense} is positive, 
then it may be true that for {\em every} del Pezzo surface $S$ of degree $1$, 
there exists such a point. Theorem \ref{maincor} is the first result that 
states sufficient conditions for the set of rational points on an arbitrary del Pezzo surface of degree $1$ to be Zariski dense. 

Moreover, if $k$ is an infinite field that is finitely generated over its ground field, then $\C_Q(5)(k)$ is infinite if and only if the curve $\C_Q(5)$ has a component that is birationally equivalent to $\P^1$ or a component of genus $1$ whose Jacobian has a point of infinite order. 
The fact that the order of a point on an elliptic curve over such a field $k$ is infinite is effectively verifiable by applying the Theorem of Nagell-Lutz to sufficiently many multiples of the point. This means that for such fields~$k$, independent of Question \ref{Q:alwaysdense}, if $S(k)$ contains a point~$Q$ satisfying the conditions of Theorem \ref{maincor}, then we can find such a point, thus reducing the verification of Zariski density of $S(k)$ to a finite computation. 

Note that if the order of $Q$ in $\E_0^\ns(k)$ is $3$ and $Q$ does not lie on six $(-1)$-curves of $S$, then the assumptions in Theorem \ref{maincor} are automatically satisfied without any further condition on $\C_Q(5)$. 
Besides verifying Zariski density of rational points on explicit surfaces, Theorem \ref{maincor} also implies the following two results.
Note that both show that our criterion is strong enough to prove Zariski density of the set of rational points
on a set of del Pezzo surfaces of degree $1$ over $\Q$ that is dense in the real analytic topology on the moduli 
space of such surfaces.

\begin{theorem}\label{denseinmodulispace}
Let $f_0, \ldots, f_4, g_0, \ldots, g_6 \in \Q$ be such that the surface $S \in \P(2,3,1,1)$ given by 
\begin{equation}\label{tobechanged}
y^2 = x^3 + \left( \sum_{i=0}^4 f_i z^iw^{4-i} \right) x + \sum_{j=0}^6 g_j z^jw^{6-j} = 0 
\end{equation}
is smooth. Then for each $\ell \in \{0,\ldots,4\}$, $m \in \{0,\ldots, 6\}$, and $\varepsilon >0$, there exist 
$\lambda,\mu \in \Q$ with $|\lambda-f_\ell|<\varepsilon$ and $|\mu - g_m|<\varepsilon$ such that the surface $S' \in \P(2,3,1,1)$ 
given by \eqref{tobechanged} with the two values $f_\ell$ and $g_m$ replaced by $\lambda$ and $\mu$, respectively, is smooth and the set $S'(\Q)$ is Zariski dense in $S'$.
\end{theorem}

\begin{theorem}\label{densewhennodal}
Suppose $k$ is an infinite field of characteristic not equal to $2$ or $3$. 
If $S$ is a del Pezzo surface of degree $1$ and the associated elliptic fibration $\E \to \P^1$ has a nodal fiber over a rational point in $\P^1$, then $S(k)$ is Zariski dense in $S$.
\end{theorem}

Our strategy to prove Theorem \ref{maincor} is to exhibit a rational map $\sigma \colon \C_Q(5) \dashrightarrow S$ such that its image has a component whose strict transform on $\E$ is a multisection of $\pi$ of infinite order (cf.~\cite{bogtsch}). In the next section, we will construct $\sigma$. 
To show that the image $\sigma(\C_Q(5))$ always has a horizontal component under the conditions in Theorem \ref{maincor}, we first choose a completion $\pC_Q(5)$ of the affine curve $\C_Q(5)$ and 
show that the added points correspond naturally to limits of the sections in $\C_Q(5)$, which allows us to show that $\sigma$ extends to the extra points in $\pC_Q(5) - \C_Q(5)$, sending them to $-4Q$ or $-5Q$ on the fiber of $\pi$ containing $Q$ (Section \ref{S:completion}). 
This allows us to characterize all cases where no component of $\pC_Q(5)$ has a horizontal image under $\sigma$ in Sections \ref{examples} and \ref{S:multisection}. 
In Section \ref{S:torsionbachange}, we show that the base change of $\pi\colon \E \to \P^1$ by a curve of genus at most $1$ has no nonzero torsion sections. Finally, we apply this to a horizontal component of $\sigma(\pC_Q(5))$ to prove all our main results in Section \ref{proofs}. 

\begin{remark}
For any explicit surface $S$ with a point $Q$, it is easy to check whether $\sigma(\C_Q(5))$ has a horizontal component, and if so, whether that component is a multisection of infinite order. Since this is indeed the 
case for some specific examples, we may already conclude that it is true for $S$ and $Q$ sufficiently general. 
\end{remark}

While we consider only surfaces given by \eqref{maineq} that are smooth, i.e., del Pezzo surfaces of degree~$1$, one could also consider {\em generalized del Pezzo surfaces} of degree~$1$, which have a birational model given by \eqref{maineq} that may have isolated rational double points. As for del Pezzo surfaces, there is a natural elliptic fibration on the blow-up of a generalized del Pezzo surface at the point corresponding to $\O$. All our results up to and including Section \ref{S:multisection} also hold for generalized  del Pezzo surfaces of degree $1$, as long as we assume that the point $Q$ does not lie on a reducible fiber, with the exception of Proposition \ref{outofhell5Q}. The proof 
of Proposition \ref{outofhell5Q} shows that there is one more singular surface that we should add to the list of examples where $\sigma(\pC_Q(5))$ is not horizontal. One can actually generalize many of our results to the case that $Q$ lies on a reducible fiber, but given the significant amount of additional computations required, this is not included in this paper. 

In the proof of Theorem \ref{densewhennodal}, we will view the family of the curves $\pC_Q(5)$ as $Q$ runs through the points on the nodal fiber as an elliptic surface. We may also consider the family of {\em all} curves $\pC_Q(5)$ as an elliptic threefold over $S$, possibly adding some extra component in some fibers to achieve flatness. This threefold has real points for any surface $S$ over $\R$;  
it would be interesting to study the Hasse principle and weak approximation for this elliptic threefold.

All computations were done using {\sc Magma} \cite{magma}. 
We thank Peter Bruin, Jean-Louis Colliot-Th\'el\`ene, Bas Edixhoven, Marc Hindry, Christian Liedtke, Bjorn Poonen, Alexei Skorobogatov, Damiano Testa, Anthony V\'arilly-Alvarado, and Bianca Viray for useful discussions. 
We thank the anonymous referee for their useful comments and suggestions. 
The first author thanks Universiteit Leiden and the Max-Planck-Institut in Bonn and the second author the Centre Interfacultaire Bernoulli in Lausanne for their hospitality and support.

\section{A family of sections}\label{multisection}

By a variety over a field we mean a separated scheme of finite type over that field. In particular, we do not assume that varieties are irreducible or reduced.  By a component we always mean an irreducible component. Curves are varieties whose components all have dimension $1$ and surfaces are varieties whose components all have dimension $2$. 

Let $k$ be a field of characteristic not equal to $2$ or $3$ and 
let $\P$ denote the weighted projective space $\P(2,3,1,1)$ over $k$ with coordinates $x,y,z,w$. 
Let $\P^1$ be the projective line over $k$ with coordinates $z,w$.
The subset $Z\subset \P$ given by $z=w=0$ contains the two singular points $(1:0:0:0)$ and $(0:1:0:0)$ of $\P$, 
so the complement $U = \P - Z$ is nonsingular. The projection $\varphi \colon \P \dashrightarrow \P^1$ onto the last two cooordinates is well defined on $U$. 
For each field extension $\ell$ of $k$, let $\C(\ell)$ denote the family of all curves $C$ in $U_\ell$, defined over $\ell$, for which the restriction $\varphi|_C \colon C \to \P_\ell^1$ is an isomorphism, that is, $\C(\ell)$ is the family of sections of $\varphi \colon U_\ell \to \P_\ell^1$.  
Whenever convenient, we will freely switch between viewing the elements of $\C(\ell)$ as curves and viewing them 
as morphisms $\P^1_\ell \to U_\ell$. 
The following lemma shows that there is an algebraic variety 
whose $\ell$-points are naturally in bijection with the curves in $\C(\ell)$. 

\begin{lemma}\label{CQA5}
For every field extension $\ell$ of $k$, there is 
a bijection $\A^7(\ell) \to \C(\ell)$ sending the point $(x_0,y_0,a,b,c,p,q)$ to the curve defined by
\begin{equation} \label{xy}
x = qz^2+pzw+x_0w^2 \qquad \mbox{and} \qquad y = cz^3+bz^2w+azw^2+y_0w^3.
\end{equation}
\end{lemma}
\begin{proof}
Without loss of generality, we assume $\ell=k$. 
Clearly, the described map is well defined and injective. To show surjectivity, let $\sigma \colon \P^1 \to U$ be 
a section of $\varphi \colon U \to \P^1$. If we set $t= z/w$, then there are polynomials
$r_1,r_2,s_1,s_2 \in k[t]$ such that $\sigma$ is given on $\A^1 \subset \P^1-\{(1:0)\}$ by 
$$
t \mapsto \left[\frac{r_1(t)}{s_1(t)} : \frac{r_2(t)}{s_2(t)} : t : 1\right].
$$
The fact that the image $C$ of $\sigma$ is contained in $U$ implies that $s_1$ and $s_2$ are constant 
and the degrees of $r_1$ and $r_2$ bounded by $2$ and $3$ respectively. 
This shows that indeed there are $x_0,y_0,a,b,c,p,q \in k$ such that $C$ is given by \eqref{xy}. 
\end{proof}

Let $f,g \in k[z,w]$ be homogeneous of degree $4$ and $6$, respectively, and let $S \subset \P$ be the surface 
given by \eqref{maineq}. The number of $(-1)$-curves on $S$ is finite. Over a separable closure $\ksep$ of $k$ 
there are $240$ such curves; those that are defined over $k$ are characterized by the following lemma. 

\begin{lemma}\label{minusone}
The curves in $\C(k)$ that are contained in $S$ are exactly the $(-1)$-curves of $S$ that are defined over $k$.
\end{lemma} 
\begin{proof}
The $(-1)$-curves are defined over a separable extension of $k$ by \cite[Theorem 1]{coombes}. 
This shows that the assumption that $k$ be perfect is not necessary in \cite[Thm. 1.2]{tonydp1},
which therefore implies that the $(-1)$-curves on $S_{\ksep}$ are exactly the curves given 
by \eqref{xy} for some $x_0,y_0,a,b,c,p,q \in \ksep$, which also follows from 
\cite[Lemma 10.9]{shiodalattices}. 
The lemma follows from taking Galois invariants.  
\end{proof}

\begin{prop}\label{CSdeg6}
For each curve $C \in \C(\ksep)$ that is not contained in $S$, we have $C \cdot S = 6$.  
\end{prop}
\begin{proof}
The equations \eqref{xy} show that $C$ has degree $6$. Also, $C$ is contained in $U$, so the intersection $C \cap S$ with the hypersurface $S$ of degree $6$ is contained in $U$, which is smooth. Therefore, intersection multiplicities are defined as usual, and the weighted analogue of B\'ezout's Theorem gives $C \cdot S = \mu^{-1}(\deg C)\cdot (\deg S)$, where $\mu=6$ is the product of the weights of $\P$. The statement follows. 
\end{proof}

The intersection $S \cap Z$ consists of the single point $\O = (1:1:0:0)$.
For any point $Q \in S(k) - \{\O\}$, and for $1 \leq n \leq 6$, we let $\C_Q(n)\subset \A^7$ denote the 
subvariety of all points whose associated curve, through the bijection of Lemma \ref{CQA5},
intersects $S$ at $Q$ with multiplicity at least $n$. 
Note that for $n=5$ this coincides with the definition of $\C_Q(5)$ in the introduction. 

Let $Q \in S(k)-\{\O\}$. After applying an automorphism of $\P^1$ (and the corresponding automorphism of $\P$), we assume without loss of generality that $\varphi(Q) = 0 = (0:1)$, say $Q = (x_0:y_0:0:1)$ for some $x_0,y_0 \in k$. The variety $\C_Q(1)\subset \A^7$ consists of the points of $\A^7$ whose first two coordinates equal $x_0$ and $y_0$, respectively, so the projection onto the last five coordinates gives an isomorphism $\C_Q(1) \to \A^5$. 
From now on we will freely use this isomorphism to identify $\C_Q(1)$ and $\A^5$ with coordinates $a,b,c,p,q$.

As in the introduction, we let $\E$ denote the blow-up of $S$ at $\O$ and $\pi\colon \E \to \P^1$ the elliptic fibration induced by the anticanonical map $\varphi|_S \colon S \dashrightarrow \P^1$. 
We will sometimes identify the fiber $\E_t$ above $t = (z_0:w_0)$ with its isomorphic image on $S$, 
equal to the intersection of $S$ 
with the hyperplane $H_t$ given by $w_0z=z_0w$ and denoted by $S_t$.
The intersection $H_t \cap S$ is given by a Weierstrass equation; in particular, all fibers 
are irreducible, and therefore all singular fibers have type $I_1$ or $II$. 
Whenever we speak of vertical or horizontal curves or of fibers on $S$ or $\E$, we refer to this fibration.
We write 
\begin{align*}
f &= f_4z^4+f_3z^3w+\dots+f_0w^4, \\ 
g &= g_6z^6+g_5z^5w+\dots + g_0w^6,
\end{align*} 
so the fiber $\E_0$ above $t = 0$, containing $Q$, is given by $y^2 = x^3 + f_0x+ g_0$. 

We can give equations for $\C_Q(n)$ inside $\C_Q(1) = \A^5$ as follows. Note that $t = z/w$ is a local parameter at the point $(0:1)$ on $\P^1$. Hence, around $Q$, the curve associated to $(a,b,c,p,q) \in \C_Q(1)$ is parametrized by 
\begin{equation}\label{paramC}
\left\{
\begin{array}{rl}
x &= qt^2+pt+x_0,\\ 
y &=ct^3+bt^2+at+y_0,\\
z & = t, \\
w &= 1.
\end{array}\right.
\end{equation}
For $0 \leq i \leq 6$, let  
$F_i \in k[a,b,c,p,q]$ be the coefficient of $t^i$ in 
\begin{equation}\label{plugitin}
  - y^2 + x^3 + f(t,1) x + g(t,1), 
\end{equation}
with $x$ and $y$ as in \eqref{paramC}. Then we have 
\begin{align}
F_0 & = 0, \nonumber \\
F_1 &=    -2y_0a + (3x_0^2 + f_0)p + f_1x_0 + g_1,\nonumber\\
F_2 &=     -a^2 - 2y_0b + 3x_0p^2 + f_1p + (3x_0^2 + f_0)q + f_2x_0 + g_2,\nonumber\\
F_3 &=     -2ab - 2y_0c + p^3 + 6x_0pq + f_2p + f_1q + f_3x_0 + g_3,\label{Fis} \\
F_4 &=     -2ac - b^2 + 3p^2q + f_3p + 3x_0q^2 + f_2q + f_4x_0 + g_4,\nonumber\\
F_5 &=     -2bc + 3pq^2 + f_4p + f_3q + g_5, \nonumber\\
F_6 &=     -c^2 + q^3 + f_4q + g_6,\nonumber
\end{align}
and the variety $\C_Q(n) \subset \C_Q(1) = \A^5$ is given by the equations $F_1 = F_2 = \ldots  = F_{n-1} = 0$. 

We define the polynomials
\begin{align}
\Phi_2 &= 4(x^3+f_0x+g_0), & \Phi_4 &= \Psi \Phi_3 - \Phi_2^2, \label{idPhis} \nonumber \\
\Psi   &= \tfrac{1}{2}\tfrac{{\rm d}}{{\rm d}x}\Phi_2, & \Phi_5 &= \Phi_2^2\Phi_4 - \Phi_3^3, \\
\Phi_3 &=3x\Phi_2-\tfrac{1}{4}\Psi^2, & \Phi_6 &= \Phi_5 - \Phi_4^2. \nonumber 
\end{align}
in $k[x]$. For every integer $j$ with $2\leq j\leq 6$, the polynomial $\Phi_j$ is the factor of
the $j$-th division polynomial of the fiber $\E_0$ that corresponds to the {\em primitive} 
$j$-torsion. In particular, the polynomials $\Phi_2, \Phi_3, \Phi_2\Phi_4$, $\Phi_5$, and 
$\Phi_2\Phi_3\Phi_6$ are the $j$-th division polynomials for $j=2,3,4,5,6$, respectively. 
For notational convenience, we set $\phi_j = \Phi_j(x_0)$ for all $j\geq 2$,
as well as 
$$
\psi = \Psi(x_0), \qquad 
h_i = (f_ix_0+g_i)\phi_2^{i-1}, \qquad 
l_i = f_i\phi_2^i - h_i\psi,
$$
for $1\leq i \leq 6$, where we set $f_5=f_6=0$. 

\begin{lemma}\label{CQ4A2}
If $y_0 \neq 0$, then the projection of $\C_Q(1) = \A^5$ onto its 
last two coordinates restricts to an isomorphism $\C_Q(4) \to \A^2$.
The inverse is given by $(p,q) \mapsto (a,b,c,p,q)$ with 
\begin{align*}
a &= \frac{\psi p+2h_1}{4y_0}, \\
b &= \frac{\psi\phi_2q + 2\phi_3p^2+2l_1p+2h_2-2h_1^2}{4y_0\phi_2},\\
c &= \frac{\zeta q+ \eta}{2y_0\phi_2^2}, \qquad \mbox{with} \\
 &\qquad \zeta = \phi_2(2\phi_3p+l_1),\\
 &\qquad \eta =-\phi_4p^3-(2h_1\phi_3+l_1\psi)p^2
 +(l_2-2h_1l_1+h_1^2\psi)p +h_3-2h_1h_2+2h_1^3.
\end{align*}
\end{lemma}
\begin{proof}
Since $F_1$ is linear in $a$, the projection of $\C_Q(1) = \A^5$ along the $a$-axis induces an isomorphism 
$\rho_1$ from $\C_Q(2)$ to $\A^4$ with coordinates $(b,c,p,q)$, of which the inverse is determined by the given expression 
for $a$. The image $\rho_1(\C_Q(3)) \subset \A^4$ has a defining equation that is linear in $b$, as $F_2$ is linear $b$ and $F_1$ is independent of $b$. 
Therefore, the projection from $\rho_1(\C_Q(2)) = \A^4$ along the $b$-axis restricts to an isomorphism 
$\rho_2$ from $\rho_1(\C_Q(3))$ to $\A^3$ with coordinates $(c,p,q)$, of which the 
inverse is determined by the given expression for $b$. 
Finally, the defining 
equation of the image $\rho_2(\rho_1(\C_Q(4))) \subset \A^3$ is linear in $c$, as $F_3$ is linear in $c$ and 
$F_1$ and $F_2$ are independent of $c$. Therefore, the projection of $\rho_2(\rho_1(\C_Q(3))) = \A^3$ along the $c$-axis restricts to an isomorphism 
$\rho_3$ from $\rho_2(\rho_1(\C_Q(4)))$ to $\A^2$ with coordinates $(p,q)$, of which the inverse determined by the given expression for $c$. The composition $\rho_3 \circ \rho_2 \circ \rho_1\colon \C_Q(4) \to \A^2$ 
is the isomorphism of the lemma.
\end{proof}

From now on we will assume $y_0 \neq 0$, or equivalently $\phi_2 \neq 0$, 
and we identify $\C_Q(4)$ and $\A^2$ with coordinates $(p,q)$ through 
the isomorphism of Lemma \ref{CQ4A2}. We may eliminate the variables
$a,b,c$ from the equation $F_4=0$; after multiplying all coefficients by $\phi_2^3$, we
find that the variety $\C_Q(5) \subset \C_Q(4) = \A^2$ is defined by 
\begin{equation}\label{CQ5}
c_1q^2 + (c_2p^2+c_3p+c_4)q = c_5p^4+c_6p^3+c_7p^2+c_8p+c_9
\end{equation}
with
\begin{align*}
c_1 &=   \phi_2^2\phi_3,\\
c_2 &=   -3\phi_2\phi_4,\\
c_3 &=   -2\phi_2(l_1\psi+2h_1\phi_3),\\
c_4 &=   \phi_2(h_1^2\psi - 2l_1h_1 + l_2),\\
c_5 &=   \phi_3^2 - \phi_4\psi,\\
c_6 &=   2l_1\phi_3 -2h_1\phi_2^2 - 4h_1\phi_4 - l_1\psi^2,\\
c_7 &=   h_1^2\psi^2 - 2(3h_1^2 - h_2)\phi_3 - (4l_1h_1 - l_2)\psi + l_1^2,\\
c_8 &=   (4h_1^3 - 2h_1h_2)\psi - 6l_1h_1^2 + 2l_1h_2 + 2l_2h_1 - l_3,\\
c_9 &=   5h_1^4 - 6h_1^2h_2 + 2h_1h_3 + h_2^2 - h_4.
\end{align*}

As we assumed that 
$y_0, \phi_2$ are nonzero and that the characteristic of $k$ is not $2$ or $3$, the vanishing of $\phi_3$ and $\phi_4$ would imply that $Q$ has both order $3$ and $4$ in 
$\E_0^\ns(k)$, which is a contradiction, so the coefficients $c_1$ and $c_2$ do not both vanish, and $\C_Q(5)$ is a curve, though not necessarily reduced or irreducible.
We will identify $\C_Q(5)$ with its image in $\A^2$ and we view the coordinates $a$, $b$, and $c$ as functions on $\C_Q(4)$ or $\C_Q(5)$, as given in Lemma \ref{CQ4A2}. 

\begin{remark}\label{F4F5F6polys}
The functions $F_4, F_5$, and $F_6$ are regular on $\C_Q(4)\isom \A^2$ and can therefore be identified with polynomials in $k[p,q]$.  
\end{remark}

\begin{remark}\label{minonecurves}
The $(-1)$-curves on $S$ going through $Q$ correspond to the points of the subscheme in $\C_Q(4) \isom \A^2$ given by $F_4=F_5=F_6=0$.  
\end{remark}

\begin{remark}\label{genulas}
A special case of Theorem \ref{maincor} is Theorem 2.1(2) of \cite{ulastwo}; 
indeed, when $f=0$ and $g$ vanishes at $(1:0) \in \P^1$, and $Q=(1:1:1:0)$, then the curve $\C_Q(5)$ is isomorphic to the curve given in Theorem 2.1(2) of \cite{ulastwo}.   
The generalizations of this theorem given in \cite[Theorems C and D]{jabara}
are also a special case of our Theorem \ref{maincor}, where one uses $Q = (0:1:1:0)$, which has order $3$ in its fiber in the case of Theorem C. The proofs of Theorems C and D in \cite{jabara} are incomplete, but they do work for surfaces $S$ that are sufficiently general.
More precisely, it is not shown that the rational function 
$T(\varrho)$ in the proof of Theorem C (and its implicit equivalent for Theorem D) is always nonconstant. 
In our geometric interpretation, this is equivalent to $\sigma(\C_Q(5))$ having a horizontal component. 
Also, there is no proof of the claim that $X(2 \cdot P_{\varrho})$ is never contained in $Q[\varrho]$ in the proof of Theorem C (and its implicit equivalent for Theorem D), which is crucial for the argument that the point $P_\varrho$ has infinite order on the elliptic curve $\E_\varrho'$.
\end{remark}

Every curve $C \in \C(k)$ corresponding to a point $P \in \C_Q(5)(k)$ and not contained in $S$,intersects $S$ with multiplicity at least $5$ at $Q$,
so by Proposition \ref{CSdeg6}, there is a unique sixth point of intersection, which is also defined over $k$.
We define a rational map 
$$
\sigma \colon \C_Q(5) \dashrightarrow S
$$
by sending $P$ to the sixth intersection point of $C$ with $S$. The map $\sigma$ is defined over $k$. By Proposition \ref{CSdeg6}, it is well defined at each point $P \in \C_Q(5)$ whose corresponding curve is not contained in $S$, and thus there are at most $240$ points $P \in \C_Q(5)$ where $\sigma$ is not well defined (see Lemma \ref{minusone} and the sentences before).  
Every horizontal component of the image of $\sigma$, or its strict transform on $\E$, yields a multisection of the elliptic fibration $\pi \colon \E \to \P^1$.

We can describe the map $\sigma$ very explicitly. The curve $C$ corresponding to $(a,b,c,p,q)\in \C_Q(5)$ is parametrized 
by \eqref{paramC}. When we substitute the expressions of \eqref{paramC} into equation \eqref{plugitin}, we obtain 
$t^5(F_5+F_6t)$, so the sixth intersection point of $C \cap S$ is given by \eqref{paramC} with $t = -F_5/F_6$.

\section{A completion of the family of sections}\label{S:completion}

We keep the notation of the previous section. In particular, the field $k$, the weighted projective space $\P=\P(2,3,1,1)$ over $k$ with coordinates $x,y,z,w$, the projective line $\P^1$ over $k$ with coordinates $z,w$, the surface $S\subset \P$, and the points $\O, Q\in S(k)$ are as before, and so are the objects that depend on them, including the elliptic fibration $\pi \colon \E \to \P^1$, the elements $\psi, \phi_j, c_i \in k$, the curve $\C_Q(5) \subset \A^2$, the coordinates $p,q$ of $\A^2$, the functions $a,b,c,F_i$ on $\C_Q(5)$, and the map $\sigma \colon \C_Q(5) \dashrightarrow S$.  

We will see in Theorem \ref{ellfibnotors} that when the closure of the image $\sigma(\C_Q(5))\subset S$ contains a horizontal component with respect to the natural elliptic fibration $\pi\colon \E \to \P^1$, then we can use such a component to construct a base change of $\pi$ with a section of infinite order. 
Unfortunately, in some cases the image $\sigma(\C_Q(5))$ does not contain such a component. 
In order to investigate when this happens, we extend the map 
$\sigma \colon \C_Q(5) \dashrightarrow S$ to a projective completion $\pC_Q(5)$ of the affine curve $\C_Q(5)$ and first determine the image of the limit points in $\Omega=\pC_Q(5)-\C_Q(5)$ (see Proposition \ref{imageinf}).  

For every extension $\ell$ of $k$, the points in $\C_Q(5)(\ell)$ correspond to elements of $\C(\ell)$, which are curves in $U_{\ell}$. So the curve $\C_Q(5)$ parametrizes a family of curves in $U\subset \P$. The elements of $\Omega$ correspond to the limit curves of this family.
Viewing the elements of $\C(\ell)$ as sections $\P_{\ell}^1 \to U_{\ell}$ of $\varphi$, we define the morphism 
$$
\gamma \colon \C_Q(5) \times \P^1 \to \P
$$ 
by $\gamma( P  ,  R ) = \chi(R)$, where $\chi \in \C(\ell)$ is the section of $\varphi$ corresponding to $P \in \C_Q(5)(\ell)$.
The morphism $\gamma$ is defined over $k$. 
In terms of the coordinates $(p,q)$ on $\C_Q(5)\subset \A^2$, the map $\gamma$ sends 
$\big((p,q),(z:w)\big)$ to $(x:y:z:w)$
with $x$ and $y$ as in \eqref{xy} and $a,b,c$ as in Lemma \ref{CQ4A2}. For each 
point $P \in\C_Q(5)(\ell)$ with corresponding section $\chi \in \C(\ell)$, 
the image $\chi(\P^1_\ell) \subset U_\ell \subset \P_\ell$ is the image 
under $\gamma$ of the fiber of the trivial $\P^1$-bundle $\C_Q(5) \times \P^1$ over $P$.
Therefore, we may find an appropriate completion $\pC_Q(5)$, as well as the 
limit curves corresponding to the elements in $\pC_Q(5)-\C_Q(5)$ as follows. 
Start with an arbitrary completion $\C_Q^0(5)$ of $\C_Q(5)$ and the trivial $\P^1$-bundle 
$\Gamma^0 = \C_Q^0(5) \times \P^1$ over it. Now $\gamma$ is defined on an open subset of $\Gamma^0$. After an appropriate sequence of blow-ups and blow-downs, we obtain a surface 
$\Gamma$ that is birational to $\Gamma^0$ to which $\gamma$ extends as a morphism, as well as a new completion $\pC_Q(5)$ such that the $\P^1$-bundle structure $\C_Q(5) \times \P^1 \to \C_Q(5)$ extends to a conic bundle structure $\Gamma \to \pC_Q(5)$. Note that it is not necessary to require that $\pC_Q(5)$ be smooth.
The limit curves are then the images under $\gamma$ of the fibers of $\Gamma \to \pC_Q(5)$ over the points in $\Omega = \pC_Q(5)-\C_Q(5)$.

The problem with the process above, in which we construct $\pC_Q(5)$ and $\Gamma$, is that we are not working with a single del Pezzo surface of degree $1$, but with all of them, and we have to distinguish several cases of monoidal transformations, based on the types of singularities at the points in $\C_Q^0(5)-\C_Q(5)$. Therefore, instead of presenting this process here, we will immediately introduce the result: a completion $\pC_Q(5)$ together with a conic bundle $\Gamma \to \pC_Q(5)$ that works in all cases, in the sense that $\gamma$ extends to it. 

\subsection{Compactifying $\pC_Q(5)$}

Let $\ovp,\ovq,\ovr$ 
be the coordinates of the weighted projective space $\P(1,2,1)$, and let $\HH \to \P(1,2,1)$ be the blow-up at the singular point $(0:1:0)$. Since $\P(1,2,1)$ is isomorphic to a cone in $\P^3$, the surface $\HH$ is smooth; it is in fact a Hirzebruch surface.  By sending $(p,q)$ to $(p:q:1)$, we identify $\A^2$ with an open subset of $\P(1,2,1)$ and hence with an open subset of $\HH$. In doing so, we also identify the function field $k(p,q)$ of $\A^2$ with that of $\HH$.  

Let $\pC_Q(5)$ denote the completion of $\C_Q(5)$ inside $\HH$. Note that the completion of $\C_Q(5)$ inside $\P(1,2,1)$ contains the singular point $(0:1:0)$ if and only if the coefficient $c_1$ of $q^2$ in \eqref{CQ5} vanishes, i.e., if and only if $Q$ has order $3$ in $\E_0^\ns(k)$. Hence, if $Q$ does not have order $3$, we may identify $\pC_Q(5)$ with the completion of $\C_Q(5)$ inside $\P(1,2,1)$; as $c_1$, $c_2$, and $c_5$ do not all vanish, this completion is given by
\begin{equation}\label{completion}
c_1\ovq^2 + (c_2\ovp^2+c_3\ovp\ovr+c_4\ovr^2)\ovq = c_5\ovp^4+c_6\ovp^3\ovr+c_7\ovp^2\ovr^2+c_8\ovp\ovr^3+c_9\ovr^4.
\end{equation}

We identify $\HH$ with the variety in $\P(1,2,1) \times \P^1(s,t)$ given by 
$\ovp t=\ovr s$. Denoting the zeroset in $\HH$ of a doubly homogeneous 
polynomial $h$ in $k[\ovp,\ovq,\ovr][s,t]$ by $Z(h)$, we define the open subsets 
$$
\HH_1 = \HH - Z(\ovr), \qquad \HH_2 = \HH - Z(\ovp), \qquad \HH_3 = \HH - Z(\ovq t), \qquad \HH_4 = \HH - Z(\ovq s)
$$ 
of $\HH$. In the function field of $\HH$, we have 
$p=\frac{s}{t}=\frac{\ovp}{\ovr}$ and $q = \frac{\ovq}{\ovr^2}$.
We define the functions 
\begin{align}
\lambda_1&=p, & \lambda_2&=p^{-1}, & \lambda_3&=p,  & \lambda_4&=p^{-1},  \nonumber \\
\mu_1 &= q , & \mu_2 & =qp^{-2} , & \mu_3 &= q^{-1} ,& \mu_4 &= p^2q^{-1} \nonumber \\
\fa_1 &= a, & \fa_2 &= a\lambda_2,   & \fa_3 &= a,      & \fa_4 &= a\lambda_4, \label{fraks} \\ 
\fb_1 &= b, & \fb_2 &= b\lambda_2^2, & \fb_3 &= b\mu_3, & \fb_4 &= b\lambda_4^2\mu_4, \nonumber \\ 
\fc_1 &= c, & \fc_2 &= c\lambda_2^3, & \fc_3 &= c\mu_3, & \fc_4 &= c\lambda_4^3\mu_4 \nonumber
\end{align}
in the function field $k(p,q)$ of $\HH$. 

\begin{lemma}\label{lem:affcoords}
For each $i \in \{1,2,3,4\}$, the functions $\lambda_i,\mu_i,\fa_i,\fb_i,\fc_i$ are 
regular on $\HH_i$ and the map $\HH_i \to \A^2$ sending $R$ to 
$(\lambda_i(R),\mu_i(R))$ is an isomorphism.
The sets $\HH_1$, $\HH_2$, $\HH_3$, and $\HH_4$ cover $\HH$. 
\end{lemma}
\begin{proof}
Suppose $i \in \{1,2,3,4\}$. The fact that $\lambda_i$ and $\mu_i$ are regular on 
$\HH_i$ and define an isomorphism to $\A^2$ is a standard computation. So is the 
last statement.
Using Lemma~\ref{CQ4A2}, one can express $\fa_i,\fb_i,\fc_i$ as polynomials
in $\lambda_i$ and $\mu_i$, which shows that $\fa_i,\fb_i,\fc_i$ are regular as well.
\end{proof}
For each $i \in \{1,2,3,4\}$, set $\C_Q^i(5) = \pC_Q(5) \cap \HH_i$. 
The union of the four affine curves $\C_Q^i(5)$, with $1 \leq i \leq 4$, is $\pC_Q(5)$. 
Note that $\HH_1 = \A^2(p,q)$ and $\C_Q^1(5) = \C_Q(5)$. 
The affine curve $\C_Q^2(5)$ coincides with the affine part with $\ovp\neq 0$ of the curve in $\P(1,2,1)$ given by \eqref{completion}; the affine coordinates $(\lambda_2,\mu_2)$ correspond with $(\ovr/\ovp,\ovq/\ovp^2)$. By abuse of notation, we will denote the restrictions of 
$\lambda_i$, $\mu_i$, $\fa_i$, $\fb_i$, and $\fc_i$ to $\C_Q^i(5)$ by the same symbol. 

\subsection{Extending $\gamma$} \label{extendinggamma}
We define the conic bundles
\begin{align*}
\Delta_1 &= \HH_1 \times \P^1(z,w), \\
\Delta_2 &= \HH_2 \times \P^1(z',w'), \\ 
\Delta_3 & \subset \HH_3 \times \P^2(u_0,u_1,u_2) \quad  \mbox{ given by } \,\,\ovr^2u_0u_2=\ovq u_1^2, \quad \mbox{and} \\
\Delta_4 & \subset \HH_4 \times \P^2(u'_0,u'_1,u'_2) \quad \mbox{ given by }\,\, \ovp^2u'_0u'_2=\ovq {u'_1}^2
\end{align*}
over $\HH_1$, $\HH_2$, $\HH_3$, and $\HH_4$, respectively.
We glue these conic bundles to a conic bundle $\Delta$ over $\HH$ as follows. 
We glue $\Delta_1$ and $\Delta_2$ above the intersection $\HH_1 \cap \HH_2$
by identifying $(z:w)\in \P^1(z,w)$ with $(\ovp z:\ovr w)\in \P^1(z',w')$.
We also glue $\Delta_1$ and $\Delta_3$ above the intersection 
$\HH_1 \cap \HH_3$ by identifying $(z:w)\in \P^1(z,w)$
with $(\ovq z^2: \ovr^2zw: \ovr^2w^2) \in \P^2(u_0,u_1,u_2)$.  
We glue $\Delta_3$ and $\Delta_4$ above $\HH_3 \cap \HH_4$ 
by identifying $(u_0:u_1:u_2)\in \P^2(u_0,u_1,u_2)$ with 
$(tu_0:su_1:tu_2)\in \P^2(u'_0,u'_1,u'_2)$.  
One easily checks that these identifications also induce an isomorphism 
between the parts of $\Delta_i$ and $\Delta_j$ above the intersection 
$\HH_i \cap \HH_j$ for the remaining pairs $(i,j) \in \{(1,4),(2,3),(2,4)\}$. 

The map $\gamma \colon \C_Q(5) \times \P^1 \to \P$ extends to 
$\C_Q(4) \times \P^1 = \A^2 \times \P^1 = \HH_1\times \P^1 =\Delta_1$
by setting $\gamma(P,R) = \chi(R)$ where, for any field extension $\ell$ of $k$, 
we have $R \in \P^1(\ell)$, and the section $\chi \in \C(\ell)$ of $\varphi$ 
corresponds to $P \in \C_Q(4)(\ell)$. 
The extended map, also denoted by  $\gamma$, sends $\big(P,(z:w)\big) \in \C_Q(4) \times \P^1$ to $(x:y:z:w)$
with $x$ and $y$ as in \eqref{xy}, with $(p,q) = (p(P), q(P)) = (\lambda_1(P),\mu_1(P))$, 
and with $a,b,c$ as in Lemma \ref{CQ4A2}.
The following proposition shows that $\gamma$ extends to a morphism 
$\Delta \to \P$. 

\begin{prop}\label{propgammaext}
The map $\gamma$ extends to a morphism $\Delta \to \P$ that is given on 
$\Delta_2$ by sending $( P, (z':w'))$ to 
$$
(\mu_2(P) z'^2+ z'w' + x_0w'^2 : \fc_2(P) z'^3+\fb_2(P)z'^2w'+\fa_2(P)z'w'^2+y_0w'^3 : \lambda_2(P) z':w'),
$$
on $\Delta_3$ by sending $(P, (u_0:u_1:u_2))$ to 
$$
(u_2(u_0+\lambda_3(P)u_1+x_0u_2):u_2(\fc_3(P)u_0u_1+\fb_3(P)u_0u_2+\fa_3(P)u_1u_2+y_0u_2^2):u_1:u_2),
$$
and on $\Delta_4$ by sending $(P, (u'_0:u'_1:u'_2))$ to 
$$
(u'_2(u'_0+u'_1+x_0u'_2) : u'_2(\fc_4(P)u'_0u'_1+\fb_4(P)u'_0u'_2+\fa_4(P)u'_1u'_2+y_0{u'_2}^2) : \lambda_4(P)u'_1 : u'_2).
$$
\end{prop}
\begin{proof}
It is easy to check that the given maps coincide with $\gamma$ wherever they are well defined. Hence, it suffices to show that they are well defined on the claimed subsets. 

Suppose the first map is not well defined at a point $( P, (z':w')) \in \Delta_2$. 
By Lemma \ref{lem:affcoords}, the functions
$\lambda_2,\mu_2,\fa_2,\fb_2,\fc_2$ are all regular at $P$, so 
the fact that the map is not well defined at $( P, (z':w'))$ implies that the four given polynomials that are claimed to define the map on $\Delta_2$ vanish. 
This yields $w'=0$, so $z'\neq 0$, and thus $\lambda_2,\mu_2$, and $\fc_2$ all vanish at $P$.  From Lemma~\ref{CQ4A2} and $\lambda_2(P)=\mu_2(P)=0$, we obtain 
$\fc_2(P) = -\phi_4/(2y_0\phi_2^2)$, so the vanishing of $\fc_2$ 
at $P$ gives $\phi_4=0$. 
From $\lambda_2(P)=\mu_2(P)=0$ and equation \eqref{completion}, we find $c_5=0$, so we also have $\phi_3^2 = \phi_4\psi=0$. This is a contradiction as $Q$ can not have both order $3$ and $4$ on $\E_0^\ns(k)$. Hence, the first map is well defined on $\Delta_2$. 

It is clear that the second map is well defined at any point 
$( P, (u_0:u_1:u_2)) \in \Delta_3$ with $u_1\neq 0$ or $u_2 \neq 0$. 
To see that it is also well defined at points with $(u_0:u_1:u_2)= (1:0:0)$, we identify $\P$ with its image under the closed immersion to $\P^{22}$ corresponding to $\O(6)$ on $\P$. Substituting the expressions for the second map into the $23$ monomials of weighted degree $6$ in the variables $x$, $y$, $z$, and $w$ gives $23$ polynomials of total degree $6$ in $u_0,u_1,u_2$, which after replacing $u_1^2$ by $\mu_3u_0u_2$ (the conic bundle $\Delta_3$ is given by $\mu_3u_0u_2=u_1^2$) are all divisible by $u_2^3$. 
The composition $\Delta_3 \to \P \to \P^{22}$ is given by these $23$ polynomials, each divided by $u_2^3$. 
The coordinate corresponding to the monomial $x^3$ is given by $(u_0+\lambda_3u_1+x_0u_2)^3$, which does not vanish at $(P,(1:0:0))$,
so this composition, and thus the map $\Delta_3 \to \P$, is well defined.  

The third map is well defined whenever $u_2'\neq 0$. On the other hand, if $u_2' = 0$, then also $u_1'=0$, and one uses the composition $\Delta_4 \to \P \to \P^{22}$ to check that the map $\Delta_4 \to \P$ is well defined at points with $(u'_0:u'_1:u'_2)= (1:0:0)$, as in the previous case. 
\end{proof}

Let $\Gamma$ be the inverse image of $\pC_Q(5)$ under the map $\Delta \to \HH$. 
We denote the restriction $\Gamma \to \pC_Q(5)$ of the conic bundle $\Delta \to \HH$ by $\tau$. By abuse of notation, we denote the restriction $\Gamma \to \P$ of the map $\gamma \colon \Delta \to \P$ by $\gamma$ as well. 

\subsection{The limit curves and their images}
Set $\Omega = \pC_Q(5) - \C_Q(5)$.  Then the limit curves described in the beginning of this section are the images under $\gamma$ of the fibers of $\tau$ above the points in $\Omega$. 
These images are described in Lemmas \ref{curvesinf}, \ref{L:E0node}, \ref{L:E0cusp}, and \ref{imagetangent}. Recall that $\C_Q^2(5)$ can be identified with the open subset of 
$\P(1,2,1)$ given by $\ovp \neq 0$.

\begin{lemma}\label{curvesinf}
Each point $P \in \C^2_Q(5)-\C_Q(5)$ corresponds to $(\ovp:\ovq:\ovr) = (1:\alpha:0)$ 
for some $\alpha \in \kbar$ satisfying $c_1\alpha^2+c_2\alpha-c_5$;
the map $\gamma$ sends $(P, (z':w')) \in \Gamma$ to $(x:y:z:w)$ with 
$$
\left\{
\begin{array}{rl}
x&=\alpha {z'}^2+ z'w' + x_0{w'}^2,\\
y&=y_0 \left( \frac{4\alpha\phi_2\phi_3-2\phi_4}{\phi_2^3}{z'}^3 +
\frac{\alpha\psi\phi_2+2\phi_3}{\phi_2^2}{z'}^2w'+
\frac{\psi}{\phi_2}z'{w'}^2 +{w'}^3  
\right),\\
z&=0,\\
w&=w',
\end{array}
\right.
$$
and the image of the fiber $\tau^{-1}(P) \subset \Gamma$ 
under $\gamma$ is a curve in $\P$ of degree $6$ that intersects $S$ at $Q$ with 
multiplicity at least $5$. 
\end{lemma}
\begin{proof}
Let $P \in \C^2_Q(5)-\C_Q(5)$.
The first part of the statement is obvious. 
We have $\lambda_2(P) =0$ and $\mu_2(P) = \alpha$. 
From Lemma \ref{CQ4A2} we deduce 
$$
\fa_2(P) = \frac{\psi}{4y_0}, \qquad \fb_2(C) = \frac{\psi\phi_2\alpha+2\phi_3}{4y_0\phi_2}, \qquad 
\fc_2(P) = \frac{2\phi_2\phi_3\alpha-\phi_4}{2y_0\phi_2^2}.
$$
Hence, according to Proposition \ref{propgammaext}, 
the map $\gamma$ sends $(P, (z':w')) \in \Delta_2$ to $(x:y:0:w')$ with 
$x =\alpha {z'}^2+z'w'+x_0{w'}^2$ and $y= \fc_2(P){z'}^3+\fb_2(P){z'}^2w'+\fa_2(P)z'{w'}^2+y_0{w'}^3$. From $4y_0^2 = \phi_2$, 
it follows that the latter equals the expression given for $y$ in the lemma. 

The curve $D=\gamma(\tau^{-1}(P))$ in $\P$ lies inside the hyperplane given by $z=0$, which is
isomorphic to the weighted projective space $\P(2,3,1)$. The intersection of $D$ with the curve $D'$ in this hyperplane given by $y=\lambda xw + \mu w^3$ yields three intersection points for general $\lambda$ and $\mu$. 
B\'ezout's Theorem tells us that the product of the weights $(2,3,1)$ times this intersection number $3$ equals $(\deg D)(\deg D')$, so we find $\deg D = 18/\deg D'=6$. 

Since the degree of $D$ is $6$, it is a full limit of images under $\gamma$ of fibers of $\tau \colon \C_Q(5) \times \P^1 \to \C_Q(5)$, all of which intersect $S$ at $Q$ with multiplicity at least $5$, so $D$ does this as well. This can also be checked computationally by substituting the parametrization given in the lemma into the polynomial 
$$
-y^2 + x^3 + f(z,w)x+g(z,w),
$$
and checking that the coefficients of ${z'}^i{w'}^{6-i}$ vanish for $0 \leq i \leq 4$.
\end{proof}

Recall that $S_0$ is the image of $\E_0$ on $S$, which is the intersection of $S$ 
with the plane given by $z=0$. 
The following two lemmas give more information about the image under $\gamma$ of the fibers of $\tau$ above points 
in $\C^2_Q(5)-\C_Q(5)$ in the case that $S_0$ is singular. In particular, they show that $S_0$ is one of the limit curves in this case. 

\begin{lemma}\label{L:E0node}
Suppose $\E_0$ has a node. Then $\C^2_Q(5) - \C_Q(5)$ contains the point
$P_1=(1:\alpha_1:0) \in \P(1,2,1)$ with
$$
\alpha_1 = \frac{f_0}{4(f_0x_0-3g_0)}.
$$
The map $\gamma$ restricts to a birational morphism from 
the fiber $\tau^{-1}(P_1)$ to $S_0$. If $\phi_3 = 0$, then $P_1$ is the only point in $\C^2_Q(5) - \C_Q(5)$.
If $\phi_3 \neq 0$, then $\C^2_Q(5) - \C_Q(5)$ contains a unique second point 
$P_2=(1:\alpha_2:0) \in \P(1,2,1)$ with 
$$
\alpha_2 = \frac{f_0(2f_0x_0-21g_0)}{4(f_0x_0-3g_0)(2f_0x_0-9g_0)};
$$ 
the image under $\gamma$ of the 
fiber $\tau^{-1}(P_2)$ is not contained in $S$.
\end{lemma}
\begin{proof}
Since $\E_0$ has a node, we have $4f_0^3+27g_0^2=0$ with $f_0,g_0 \neq 0$, so for 
$d = -\frac{3g_0}{2f_0}$ we have $f_0 = -3d^2$ and $g_0 = 2d^3$. The curve $\E_0\isom S_0$ is given 
by $y^2 = (x-d)^2(x+2d)$, and we have 
\begin{align*}
\Phi_2 &= 4(x-d)^2(x+2d), \\
\Phi_3 &= 3(x-d)^3(x+3d), \\
\Phi_4 &= 2(x-d)^5(x+5d), \\
\Phi_5 &= (x-d)^{10}(5x^2+50dx+89d^2).
\end{align*}
If $\phi_3 \neq 0$, then $c_1 \neq 0$ and the polynomial $c_1T^2+c_2T-c_5$ 
factors as $c_1(T-\alpha_1)(T-\alpha_2)$ with $\alpha_1=\frac{1}{4}(x_0+2d)^{-1}$ and 
$\alpha_2 = \frac{1}{4}(x_0+7d)(x_0+2d)^{-1}(x_0+3d)^{-1}$, which equal the expressions 
given in the proposition. If $\phi_3=0$, then $c_1=0$ and $x_0=-3d$, so the only root of 
$c_1T^2+c_2T-c_5$ is $\alpha_1 = c_5/c_2 = -\frac{1}{4}d^{-1}$, which equals the expression for $\alpha_1$ 
given in the proposition. It follows from Lemma \ref{curvesinf} that the 
points in $\C^2_Q(5) - \C_Q(5)$ are as claimed. 

It follows from Lemma \ref{curvesinf} and the identities above that the restriction 
of $\gamma$ to $\tau^{-1}(P_1) = \{P_1\} \times \P^1(z',w')$ factors as the composition of the isomorphism 
$$
\{P_1\} \times \P^1(z',w') \to \P^1, \quad (P_1,(z':w')) \mapsto ((x_0-d)(z'+2(x_0+2d)w') : 2y_0w')
$$ 
and the birational morphism 
$$
\P^1 \to S_0, \quad (s:1) \mapsto (s^2-2d : s^3-3ds: 0: 1).
$$
This proves the second statement.

For the last statement, we assume $\phi_3 \neq 0$, take $\alpha = \alpha_2$ and substitute the 
corresponding parametrization of Lemma \ref{curvesinf} in the equation 
$$
-y^2 + x^3 +f(z,w)x+g(z,w) =0, 
$$  
which defines $S$. The obtained equation in $z'$ and $w'$, multiplied by 
$$-16d^{-3}(x_0-d)^{10}(x_0+2d)^5(x_0+3d)^3,$$
is
$$
z'^5\big(\phi_5\cdot z'+(x_0-d)^6\phi_2\phi_3 \cdot w'\big) = 0.
$$
As the left-hand side does not vanish identically, 
the curve $\gamma(\tau^{-1}(P_2))$ is not contained in $S$.
\end{proof}

\begin{lemma}\label{L:E0cusp}
Suppose that $\E_0$ has a cusp. Then $\pC_Q(5)$ equals $\C_Q(5) \cup \C^2_Q(5) = \C_Q^1(5) \cup \C^2_Q(5)$ and $\C^2_Q(5) - \C_Q(5)$ contains exactly one point, namely 
$P = (2:x_0^{-1}:0) \in \P(1,2,1)$.
The map $\gamma$ restricts to a birational morphism from the fiber $\tau^{-1}(P)$ to $S_0$. 
\end{lemma}
\begin{proof}
Since $\E_0$, or equivalently $S_0$, has a cusp, we have $f_0= g_0=0$. The cusp $(0:0:0:1)$ is the only point on $S_0$ with $x$-coordinate $0$, so from $y_0 \neq 0$ we find $x_0 \neq 0$. 
The $\Phi_i$ are as in the proof of Lemma \ref{L:E0node} with $d=0$. 
From $c_1 = 48x_0^{10} \neq 0$ we get $\pC_Q(5) = \C_Q(5) \cup \C^2_Q(5)$.
The polynomial $c_1T^2+c_2T-c_5$ factors as 
$3x_0^8(4x_0T-1)^2$ with the unique root $\alpha = (4x_0)^{-1}$, which 
implies by Lemma \ref{curvesinf} that $P = (2: x_0^{-1}:0)$ is the only point in $\C^2_Q(5) - \C_Q(5)$. 
One checks by a computation that it also follows from Lemma \ref{curvesinf} that 
the restriction of $\gamma$ to $\tau^{-1}(P) = \{P\} \times \P^1(z',w')$ factors as the composition of the isomorphism 
$$
\{P\} \times \P^1(z',w') \to \P^1, \quad (P,(z':w')) \mapsto (z'+2x_0w' : 2x_0w')
$$ 
and the birational morphism 
$\P^1 \to S_0$ that sends $(s:1)$ to $(x_0s^2 : y_0s^3: 0: 1)$. 
This proves the second statement.
\end{proof}

The points of $\pC_Q(5) - \C_Q(5)$ that are not handled by the previous lemmas are the points 
in $\pC_Q(5)  - \big(\C^1_Q(5)\cup \C_Q^2(5)\big)$, that is, the points 
above the singular point $(0:1:0)$ in $\P(1,2,1)$. The next lemma takes care of these points. 

\begin{lemma}\label{imagetangent}
For each point $P \in \pC_Q(5)  - \big(\C_Q^1(5)\cup \C_Q^2(5)\big)$,
the map $\gamma$ sends the fiber $\tau^{-1}(P)$ to the curve in $\P$ given by $z=0$
and $4y_0y=\psi xw + (\phi_2-\psi x_0)w^3$; this curve intersects 
$S$ with multiplicity $3$ at $Q$ and nowhere else. 
\end{lemma}
\begin{proof}
By Lemma \ref{lem:affcoords}, the open subset $\HH_3$ has affine coordinates $(\lambda_3,\mu_3)$.
If $P$ lies in $\C^3_Q(5)  - \big(\C^1_Q(5)\cup \C_Q^2(5)\big)$, then it corresponds with a point with $(\lambda_3,\mu_3) = (p,0)$ for some $p\in \kbar$ and the fiber  
$\tau^{-1}(P)$ is given by $u_1^2=0$ in $\P^2(u_0,u_1,u_2)$. The map  
$\gamma$ sends $(P,(u:0:1))\in \Delta_3$ to $(u+x_0:\fb_3(P)u+y_0:0:1)$ by Lemma \ref{propgammaext}.
From Lemma \ref{CQ4A2}, we find $4y_0\fb_3(P) = \psi$. 
It follows that the image of the fiber is the claimed curve. 
In the affine plane given by $z=0$ and $w=1$, 
this curve is a line going through $Q$ with slope 
$\fb_3(P) = (3x_0^2+f_0)/(2y_0)$, so it is exactly the tangent line to 
the curve $S_0$ at $Q$. 
Note that the existence of $P$ implies 
that $Q$ has order $3$ on $S_0^\ns(k)$, so this tangent line intersects $S_0$ with multiplicity $3$ 
at $Q$ and nowhere else. As the curve intersects the surface $S$ only in the curve $S_0$, 
the lemma follows.  If $P$ lies in $\C^4_Q(5) - \big(\C^1_Q(5)\cup \C^2_Q(5)\big)$, then the argument is analogous.
\end{proof}

\begin{remark}\label{nonreducedinf}
Scheme theoretically, the image under $\gamma$ of the fiber of $\tau$ above $P$ in Lemma \ref{imagetangent}
is not reduced, but given by $z^2=0$ and $4y_0y=\psi xw + (\phi_2-\psi x_0)w^3$. This nonreduced curve is also a limit curve as mentioned in the beginning of the section, 
and it intersects $S$ with multiplicity $6$ at $Q$. 
\end{remark}

\begin{remark}\label{Tremark}
Let $T$ be the image of $\gamma \colon \Gamma \to \P$. Then $T$ is the union of all curves $C\subset U$ 
corresponding to points $P \in \C_Q(5)$ and the limit curves corresponding to points $P \in \Omega$. 
The closure of the image $\sigma(\C_Q(5))$ in $S$ is contained in the intersection $S \cap T$. This intersection 
also contains all $(-1)$-curves on $S$ that go through $Q$. See also Remarks \ref{whatisthiscurve} and \ref{finalremarkchapter5}.
\end{remark}

The rational map $\sigma \colon \C_Q(5) \dashrightarrow S$ from the end of Section 
\ref{multisection} factors as $\sigma = \gamma \circ \rho$, where $\rho \colon 
\C_Q(5) \dashrightarrow \C_Q(5) \times \P^1(z,w)$ is a rational section of 
$\tau \colon \Gamma \to \pC_Q(5)$ that sends $P\in \C_Q(5)$ to $\big(P,(-F_5(P) : F_6(P))\big)$. 
Here, for $0\leq i \leq 6$, we view $F_i$ as in \eqref{Fis} as a function on $\pC_Q(5)$. 
We use this in Proposition \ref{imageinf} to show that $\sigma$ extends to a map that is 
well defined at every point in $\Omega = \pC_Q(5) - \C_Q(5)$.

$$
\xymatrix{
\pC_Q(5) & \Gamma \ar[l]_(0.4){\tau} \ar[r]^\gamma & \P &S \ar@{_(->}[l] \\
\C_Q(5) \ar@{^(->}[u] \ar@/_5.5mm/@{-->}[rrr]_{\sigma}\ar@/^5.5mm/@{-->}[r]^{\rho} & 
                \C_Q(5) \times \P^1 \ar@{^(->}[u]\ar[l] \ar[r]^(0.63)\gamma & U \ar@{^(->}[u] & 
                                                            S-\{\O\} \ar@{_(->}[l]\ar@{^(->}[u]
}
$$

\begin{prop}\label{imageinf}
The rational map $\sigma$ extends to a rational map $\pC_Q(5) \dashrightarrow S$ 
that is well defined at the points in $\Omega$. For every $P \in \Omega$, we have 
$\sigma(P)=-4Q \in S_0^\ns(k) \subset S$ if $S_0$ has a cusp or $S_0$ has a node and $P=P_1$ as in Lemma \ref{L:E0node}, and we have $\sigma(P)=-5Q \in S_0^\ns(k) \subset S$ otherwise.
\end{prop}

\begin{proof}
Let $P \in \Omega$. 
Then by Lemmas \ref{curvesinf}, \ref{L:E0node}, \ref{L:E0cusp}, \ref{imagetangent},
and Remark \ref{nonreducedinf},
the scheme-theoretic image of $\tau^{-1}(P)$ under $\gamma$ is a curve 
of degree $6$ in the plane given by $z=0$ in $\P$.
We denote this curve by $C$. 
A parametrization of $C$ is given in Lemma \ref{curvesinf} if 
$P \in \C^2_Q(5) - \C_Q(5)$; the curve $C$ is nonreduced if 
$P \in \pC_Q(5) - \big(\C_Q^1(5) \cup \C^2_Q(5)\big)$.
The intersection $C \cap S$ is the same as the intersection of $C$ with 
with $S_0 = S \cap \{z=0\}$, and $C$ intersects $S_0$ with multiplicity at least $5$ 
at $Q$. 

If $S_0$ is smooth, then $S_0$ has genus 
$1$, so $C$ has no components in common with $S_0$. 
The curves $S_0$ and $C$ also 
have no components in common if  
$P \in \pC_Q(5) - \big(\C_Q^1(5) \cup \C^2_Q(5)\big)$ (Lemma \ref{imagetangent}
and Remark \ref{nonreducedinf})
or $P=P_2$ as in Lemma \ref{L:E0node}. Hence, in all these cases 
there is a unique 
sixth intersection point in $C\cap S = C \cap S_0$, and we can extend $\sigma$ to 
$P$ by sending $P$ to this sixth intersection point, say $R$; the divisor 
$5(Q)+(R)$ on $S_0$ is a hypersurface section inside the plane given by $z=0$,
so it is linearly equivalent to a multiple of $3(\O \cap S_0)$ on $S_0$, 
and we find $R = -5Q$ in $S_0^\ns(k)$. 

We are left with the case that $S_0$ has a cusp (Lemma \ref{L:E0cusp}), or $S_0$ has a node and $P=P_1$ as in Lemma 
\ref{L:E0node}. In both cases, there is a $d \in k$ such that $f_0 = -3d^2$ and $g_0=2d^3$ and, 
in terms of the coordinates $(\overline{p}:\overline{q} : \overline{r})$ on $\P(1,2,1)$, we have 
$P = (1: \alpha: 0) \in \C^2_Q(5)$ with $\alpha = \frac{1}{4}(x_0+2d)^{-1}$. 
By Lemma~\ref{lem:affcoords}, the functions $\lambda_2= \overline{r}/\overline{p}$ and $\mu_2 = \overline{q}/\overline{p}^2$ are affine coordinates for $\HH_2$, with $P$ corresponding to $(\lambda_2,\mu_2) = (0,\alpha)$, and the functions
$\fa_2,\fb_2$, and $\fc_2$ are regular on $\HH_2$. As before, we denote the restrictions of 
$\lambda_2$, $\mu_2$, $\fa_2,\fb_2$, and $\fc_2$ to $\C_Q^2(5)$ by the same symbols.

Using \eqref{Fis} and \eqref{fraks}, we can express, for each $i$, the function 
$F_i' = \lambda_2^iF_i$ on $\pC_Q(5)$ as a polynomial in terms of $\lambda_2$, $\mu_2$,
$\fa_2,\fb_2$, and $\fc_2$, which shows that $F_i'$ 
is regular on $\C_Q^2(5)$. In particular, we have
\begin{align*}
F_5' &=     -2\fb_2\fc_2 + 3\mu_2^2 + f_4\lambda_2^4 + f_3\lambda_2^3\mu_2 + g_5\lambda_2^5, \\
F_6' &=     -\fc_2^2 + \mu_2^3 + f_4\lambda_2^4\mu_2 + g_6\lambda_2^6.
\end{align*}
Recall from Subsection \ref{extendinggamma} that $\Delta_1$ and $\Delta_2$ are glued 
by setting $(z':w') = (\overline{p}z : \overline{r}w) = (z: \lambda_2 w)$.
Hence, on $\C^2_Q(5)$, the rational map $\rho \colon \C^2_Q(5) \dashrightarrow 
\C^2_Q(5) \times \P^1(z',w') \subset \Gamma$ is given 
by 
$$
\rho(P) = \big(P,(-F_5(P) : \lambda_2(P)F_6(P)) \big) = \big(P,(-F_5'(P): F_6'(P))\big).
$$
The functions $\lambda_2$ and $\mu_2-\alpha$ are local parameters for 
$\HH_2$ at $P$, so their restrictions
generate the maximal ideal $\m$ of the local ring $A_P$ of $\C_Q^2(5)$ at $P$.
From \eqref{completion}, we find that in $A_P$ we have  
\begin{equation}\label{modr2}
c_1 {\mu_2}^2 + c_2\mu_2-c_5 \equiv (c_6-c_3\mu_2)\lambda_2 
\end{equation}
modulo $\lambda_2^2$. 

Now suppose $\phi_3 \neq 0$. Then $c_1 \neq 0$ and the left-hand side of \eqref{modr2} factors as $c_1(\mu_2-\alpha)(\mu_2-\alpha')$ with 
$\alpha' = \frac{1}{4}(x_0+7d)(x_0+2d)^{-1}(x_0+3d)^{-1}$. 
In fact, $\alpha$ and $\alpha'$ correspond to $\alpha_1$ and $\alpha_2$ of Proposition~\ref{L:E0node}.
Modulo $\m^2$, the left- and right-hand side of \eqref{modr2} are congruent to 
$c_1(\mu_2-\alpha)(\alpha-\alpha')$ and $(c_6-c_3\alpha)\lambda_2$, respectively. 
Assume $d \neq 0$ as well. Then $\alpha' \neq \alpha$, 
so we find that modulo $\m^2$ we have 
$\mu_2-\alpha \equiv \delta \lambda_2$ with
$$
\delta = \frac{c_6-c_3\alpha}{c_1(\alpha-\alpha')}.
$$
Hence, $\m$ is generated by $\lambda_2$ and one checks, preferably with the help of 
a computer, that we have 
\begin{equation}\label{F5F6modr2}
F_5' \equiv \frac{(f_1d+g_1)\phi_5}{(x_0-d)^{10}\phi_2^2} \cdot \lambda_2
\qquad \mbox{and} \qquad
F_6' \equiv \frac{(f_1d+g_1)\phi_4}{(x_0-d)^5\phi_2^2} \cdot \lambda_2 \qquad \pmod{\lambda_2\m}.
\end{equation}
We claim that \eqref{F5F6modr2} also holds when $d=0$ or $\phi_3=0$. Indeed, if $\phi_3=0$, then one uses $x_0=-3d$, while $c_1=0$ and $c_2 \neq 0$, so \eqref{modr2} yields $\mu_2-\alpha \equiv c_2^{-1}(c_6-\alpha c_3)\lambda_2 \pmod{\m^2}$; it follows that  $\lambda_2$ generates $\m$, and one checks \eqref{F5F6modr2} again by computer.
If $d=0$, then $\m$ may not be principal, so being congruent modulo $\lambda_2\m$ is potentially stronger than being congruent modulo $\m^2$; but using that modulo $\lambda_2\m$ we have \eqref{modr2} and $\mu_2\lambda_2 \equiv \alpha \lambda_2$, one can again  check that \eqref{F5F6modr2} holds. Hence, \eqref{F5F6modr2} holds in all cases.

Now $f_1d+g_1$ is nonzero because the surface 
$S$ is smooth at the singular point of $S_0$. 
Also, since $Q$ is not the singular point of $S_0$, 
we have $x_0 \neq d$ and $\phi_4$ and $\phi_5$ 
do not both vanish. We conclude that $\rho \colon \C^2_Q(5) \dashrightarrow \C^2_Q(5) \times \P^1(z',w')$
is well defined at $P$, sending $P$ to $\big(P,(-F_5'(P): F_6'(P))\big) = \big(P,(-\phi_5 : (x_0-d)^{5}\phi_4)\big)$. 
Substituting this into the parametrization of Lemma \ref{curvesinf}, we find $\sigma(P) = \gamma(\rho(P)) = (x_1:y_1:0:1)$, with 
$$
x_1 = d+\frac{(x_0-d)^4}{16(x_0+2d)(x_0+5d)^2}  \qquad \mbox{and} \qquad y_1 = -\frac{(x_0-d)^3(x_0^2+22dx_0+49d^2)y_0}{64(x_0+2d)^2(x_0+5d)^3}.
$$
It is easy to check that  this point equals $-4Q$ in the group $S_0^\ns(k)$, using the fact that 
the tangent line to $S_0$ at $Q$ intersects $S_0$ also in $-2Q$, the tangent line to $S_0$ at $-2Q$ intersects $S_0$ also in $4Q$, and the inverse of a point is obtained by negating the $y$-coordinate.
\end{proof}

\begin{cor}\label{cor310}
The following statements hold. 
The multiples of $Q$ are taken in the group $S_0^{\ns}(k)$. 
\begin{enumerate}
\item We have $\sigma(\Omega) =\{-5Q\}$ if and only if $S_0$ is smooth.
\item If $\sigma(\Omega) =\{-4Q,-5Q\}$, then $S_0$ is nodal. The converse holds if $3Q\neq \O$. 
\item If $S_0$ is cuspidal, then $\sigma(\Omega) =\{-4Q\}$. The converse holds if $3Q\neq \O$. 
\item If $4Q\neq \O$ and $5Q \neq \O$, then $\sigma(\Omega) \subset S_0^{\ns}(k)-\{\O\}$ and
$\varphi(\sigma(\Omega)) = \{(0:1)\}$.
\end{enumerate}
\end{cor}
\begin{proof} 
The `if'-part of (1) follows immediately from Proposition \ref{imageinf}. 
For the `only if'-part, note that if $S_0$ is singular, 
then by Proposition \ref{imageinf} there exists a $P \in \Omega$ with $\sigma(P) = -4Q$: when $S_0$ is cuspidal, this holds for any $P \in \Omega$ and when $S_0$ is nodal, we can take $P=P_1$ as in Lemma \ref{L:E0node}. 

The first part of (3) follows directly from Proposition \ref{imageinf}. Together with (1), this also implies the first part of (2). If $3Q\neq \O$ and $S_0$ is nodal, then for the points $P_1$ and $P_2$ as in Lemma \ref{L:E0node}, we have $\sigma(P_1)=-4Q$ and $\sigma(P_2)=-5Q$ by Proposition \ref{imageinf}, which proves the second part of (2). The second part of (3) now follows from (1) and (2).

Statement (4) follows immediately from Proposition \ref{imageinf}.
\end{proof}

The next two sections investigate the conditions under which $\sigma \colon \pC_Q(5) \dashrightarrow S$ sends an irreducible component of $\pC_Q(5)$ to a fiber of $\varphi|_S \colon S \dashrightarrow \P^1$. The following lemma shows that if this is the fiber that contains $Q$, then $\sigma$ is the constant map to $Q$ on the component. 

\begin{lemma}\label{Qabove0}
Let $\C_0$ be a component of $\pC_Q(5)$ for which $\varphi(\sigma(\C_0)) = (0:1)$. Then $\sigma(\C_0)=Q$. 
\end{lemma}
\begin{proof}
Without loss of generality, we assume $k$ is algebraically closed. 
Let $P \in \C_0 \cap \C_Q(5)$ be such that the associated section $C\in \C(k)$ is not entirely 
contained in $S$. Then $\sigma$ is well defined at $P$ and $\sigma(P)$ is the unique sixth 
intersection point of $C$ with $S$. Since $C$ is a section of $\varphi \colon U \dashrightarrow \P^1$,
it intersects the fiber $S_0$ only once, namely in $Q$, and as this 
sixth intersection point lies in $S_0$ as well, we conclude $\sigma(P) = Q$. Thus all but finitely many 
points of $\C_0$ map to $Q$ under $\sigma$, so $\sigma(\C_0)=Q$.
\end{proof}

\section{Examples}\label{examples}

In this section, $k$ still denotes a field of characteristic not equal to $2$ or $3$. We will give examples of surfaces $S \subset \P$ over $k$ given by \eqref{maineq}, together with a point $Q\in S(k)$ for which the map $\sigma \colon \pC_Q(5) \dashrightarrow S$ sends at least one irreducible component of $\pC_Q(5)$ to a fiber of $\varphi|_S \colon S \dashrightarrow \P^1$. In the next section we will see that, at least outside characteristic $5$, these examples include all cases where every component of $\pC_Q(5)$ is sent to a fiber on $S$. 

In view of Theorem \ref{maincor}, it is important to note that in all examples, there are at least six $(-1)$-curves on $S$ going through $Q$. Recall from Remark \ref{minonecurves} that these correspond to the points of $\A^2\isom \C_Q(4)$ with $F_4=F_5=F_6=0$. Recall also, from the last paragraph of Section \ref{multisection}, that the map $\varphi \circ \sigma \colon \C_Q(5) \to \P^1$ is given by 
$[-F_5 : F_6]$.

\begin{example}\label{ex:hell5Q}
Let $\beta,\delta \in k^*$ and assume the characteristic of $k$ is not $5$. Set 
\begin{align*}
x_0&=3(\beta^2+6\beta+1), &f_0&=-27(\beta^4+12\beta^3+14\beta^2-12\beta+1), \\
y_0&=108\beta, & g_0&=54(\beta^2+1)(\beta^4+18\beta^3+74\beta^2-18\beta+1),
\end{align*}
and let $S \subset \P$ be the surface given by \eqref{maineq} with 
$f = f_0w^4$ and $g = \delta z^5w+g_0w^6$, and with point $Q = (x_0:y_0:0:1)$. 
Assume that $S$ is smooth, so that it is a del Pezzo surface.
The curve $S_0$ is nonsingular if and only if $\beta(\beta^2+11\beta-1) \neq 0$. The point
$Q$ has order $5$ in $S_0^\ns(k)$. 
Generically, in particular over a field in which $\beta$ and $\delta$ are 
independent transcendentals, the surface $S$ is smooth and the fibration 
$\pi \colon \E \to \P^1$ has $10$ nodal fibers (type $I_1$) and one cuspidal fiber 
(type $II$) above $(z:w)=(1:0)$. 

Let $\alpha$ be an element in a field extension of $k$ satisfying $\alpha^2 = \alpha+1$.   Then 
$\pC_Q(5)$ splits over $k(\alpha)$ into two components. 
The function $F_6$ vanishes on $\pC_Q(5)$ and the map $\sigma \colon \C_Q(5) \dashrightarrow S$ sends each component birationally to the cuspidal fiber. 
The conic bundle $\Gamma$ splits into two components as well. 
Both components of the image $T$ of 
$\gamma \colon \Gamma \to \P$ (cf. Remark \ref{Tremark}) intersect $S$ in 
the cuspidal fiber and, over an extension of $k(\alpha)$ of degree at most $5$, five $(-1)$-curves; 
the surface $T$ intersects $S$ doubly in the cuspidal curve, as well as in ten 
$(-1)$-curves going through $Q$, corresponding to the points on the affine part $\C_Q(5)$ where 
$F_5$ vanishes. 
Indeed, if $\alpha, \epsilon$ in an extension of $k$ satisfy
$$
\alpha^2 = \alpha + 1 \qquad \mbox{and} \qquad \delta =  -6(\beta + \alpha^5)\epsilon^5,
$$
then we have a section over $k(\alpha, \epsilon)$ going through $Q$ with 
\begin{align*}
x =& \epsilon^2 z^2 + 6\alpha \epsilon zw + x_0w^2, \\
y =& -\epsilon^3 z^3 + 3(\beta +2\alpha +3)\epsilon^2 z^2w      
+ 18\alpha(\beta+1) \epsilon zw^2 + y_0 w^3. 
\end{align*}
\end{example}

\begin{example}\label{ex:char5}
Let $k$ be a field of characteristic $5$ containing elements $\alpha,\beta \in k$. 
Let $S \subset \P$ be the surface given by \eqref{maineq} with 
$f = \alpha z^4$ and $g = \beta z^6 + (3\alpha+1) z^5w+ zw^5$, and with point 
$Q = (1:1:0:1)$. Assume that $S$ is smooth, so that it is a del Pezzo surface.
Generically, and in particular when $\alpha$ and $\beta$ are independent 
transcendentals, this is the case, and the fibration 
$\pi \colon \E \to \P^1$ has $10$ nodal fibers (type $I_1$) and one cuspidal fiber 
(type $II$), namely $S_0$. The curve $\C_Q(5)$ is given by
$$
q^2 + (2p^2 -1)q + p^4 - p^2 + 3\alpha = 0,
$$
and $F_5$ vanishes on $\C_Q(5)$. By Lemma \ref{Qabove0}, 
the map $\sigma$ is constant and sends $\C_Q(5)$ to $Q$. 
Generically, the curve $\C_Q(5)$ is geometrically irreducible.
There are at least ten $(-1)$-curves going through~$Q$. 
\end{example}

\begin{example}\label{ex:order3}
For any $\beta\neq 0$, the point $(x_0,y_0)=(3,\beta)$ has order $3$ on the Weierstrass curve given by 
$y^2 = x^3+f_0x+g_0$ with $f_0 = 6\beta-27$ and $g_0 = \beta^2-18\beta+54$; this curve is nonsingular
if and only if $\beta \neq 4$. 

{\em \noindent Subexample (i).} For any $\alpha_1,\alpha_2,\alpha_3 \in k$ we consider the surface $S\subset \P$ given by 
\eqref{maineq} with   
\begin{align*}
f &= -3\alpha_1^2z^4 + 3\alpha_2z^3w + (18-3\beta)\alpha_1 z^2w^2 +f_0w^4,\\
g &= \alpha_3z^6 + 3\alpha_1\alpha_2 z^5w + (18-6\beta )\alpha_1^2 z^4w^2 + (\beta - 9)\alpha_2z^3w^3 + 
(15\beta - 54)\alpha_1 z^2w^4 + g_0w^6,
\end{align*} 
and with $Q=(3:\beta:0:1)$, so that $Q$ has order $3$ on $S_0^\ns(k)$. 
Assume $S$ is smooth, so that it is a del Pezzo surface.
The affine part $\C_Q(5)$ of the curve $\pC_Q(5)$ is given by 
$$
(p^2-\beta \alpha_1)(\beta q - p^2 + 2\beta\alpha_1)=0.
$$
The function $F_5=3\beta^{-1}p(q+\alpha_1)(\beta q - p^2 + 2\beta\alpha_1)$ vanishes on the
component given by the vanishing of the second factor; by Lemma \ref{Qabove0}, this 
component is contracted by the map 
$\sigma \colon \pC_Q(5) \dashrightarrow S$, which sends it to $Q$.
There are at least six $(-1)$-curves on $S$ going through $Q$. 

{\em \noindent Subexample (ii).}
For any $\alpha_4,\alpha_5,\alpha_6 \in k$ 
we consider the surface $S\subset \P$ given by \eqref{maineq} with 
\begin{align*}
f &= 3\alpha_4z^3w +f_0w^4,\\
g &= \alpha_6z^6 + \alpha_5 z^3w^3 + g_0w^6,
\end{align*} 
and with $Q=(3:\beta:0:1)$. 
Assume $S$ is smooth, so that it is a del Pezzo surface.
The affine part $\C_Q(5)$ of the curve $\pC_Q(5)$ is given by
$$
p(\beta pq - p^3 + (\beta - 9)\alpha_4 - \alpha_5)=0.
$$
Again, the function $F_5=3\beta^{-1}q(\beta pq - p^3 + (\beta - 9)\alpha_4 - \alpha_5)$ vanishes on the component given by the vanishing of the second factor;  again by Lemma \ref{Qabove0}, this component is contracted by the map $\sigma \colon \pC_Q(5) \dashrightarrow S$, which sends it to $Q$. 
There are at least nine $(-1)$-curves on $S$ going through $Q$. 

{\em \noindent Subexample (iii).}  
Let $S$ be any smooth surface that fits in both families of these examples, i.e., with 
$\alpha_1 =0$, $\alpha_4=\alpha_2$, $\alpha_5 = (\beta-9)\alpha_2$, and $\alpha_6=\alpha_3$.  
Writing $\epsilon =\alpha_2$ and $\delta=\alpha_3$, we have 
\begin{align*}
f &= 3\epsilon z^3w+f_0w^4,\\
g &= \delta z^6+ (\beta-9)\epsilon z^3w^3 + g_0w^6.
\end{align*} 
Generically, say over a field in which $\beta, \delta$, and $\epsilon$ are independent transcendentals, the
surface $S$ is smooth and the fibration $\pi \colon \E \to \P^1$ has twelve nodal fibers. 
Suppose $S$ is indeed smooth. Then $\beta \not\in \{0,4\}$.
The affine part $\C_Q(5)$ of the curve $\pC_Q(5)$ is given by
$$
p^2 (\beta q - p^2)=0,
$$ 
so it consists of two components. The function $F_5$ vanishes on 
both components, so by Lemma \ref{Qabove0}, they are 
contracted to $Q$ by $\sigma \colon \C_Q(5) \dashrightarrow S$.  
There are at least nine $(-1)$-curves on $S$ going through~$Q$.
\end{example}

\begin{example}\label{ex:order3sec}
For any $\beta \in k^*$, the point $(0,\beta)$ has order $3$ on the elliptic curve given by $y^2 = x^3+\beta^2$. In the following three subexamples, we take $g = \epsilon z^6 + \delta z^3w^3 + \beta^2 w^6$ for some $\delta,\epsilon \in k$ and the point $Q = (0:\beta:0:1) \in \P$, which in all cases has order $3$ on $S_0$. 

{\em \noindent Subexample (i).} Let $S$ be the surface given by \eqref{maineq} with $f = \alpha z^2w^2$ for some $\alpha \in k$ and assume that $S$ is smooth. The affine part $\C_Q(5)$ of the curve $\pC_Q(5)$ is given by 
$(3p^2+\alpha)q=0$. The function $F_5=3pq^2$ vanishes on the
component given by $q=0$; by Lemma \ref{Qabove0}, this component is contracted by the map 
$\sigma \colon \pC_Q(5) \dashrightarrow S$, which sends it to $Q$.
There are at least six $(-1)$-curves on $S$ going through $Q$. Generically, there are twelve nodal fibers.

{\em \noindent Subexample (ii).} Let $S$ be the surface given by \eqref{maineq} with $f = \alpha z^3w$ for some $\alpha \in k$ and assume that $S$ is smooth. The affine part $\C_Q(5)$ of the curve $\pC_Q(5)$ is given by 
$p(3pq+\alpha)=0$. The function $F_5=q(3pq+\alpha)$ vanishes on one of the components; by Lemma \ref{Qabove0}, this component is contracted by the map 
$\sigma \colon \pC_Q(5) \dashrightarrow S$, which sends it to $Q$.
There are at least nine $(-1)$-curves on $S$ going through $Q$.  Generically, there are twelve nodal fibers.

{\em \noindent Subexample (iii).} Let $S$ be the surface given by \eqref{maineq} with $f = 0$ and assume that $S$ is smooth. The affine part $\C_Q(5)$ of the curve $\pC_Q(5)$ is given by 
$p^2q=0$. The function $F_5$ vanishes on both components, so we have $\sigma(\pC_Q(5))=Q$
by Lemma \ref{Qabove0}. The surface is isotrivial; all fibers have $j$-invariant $0$. There are at least nine $(-1)$-curves on $S$ going through $Q$, and there are six cuspidal fibers.
\end{example}

\section{A multisection}\label{S:multisection}

We continue the notation of Sections \ref{multisection} and \ref{S:completion}. In particular, the field $k$ with characteristic not equal to $2$ or $3$, the surface $S$, and the point $Q$ are fixed as before, as are all the objects that depend on them. 

As we have seen in the previous section, not every component of $\pC_Q(5)$ necessarily has its image under $\sigma \colon \pC_Q(5) \dashrightarrow S$ map dominantly to $\P^1$ under the projection $\varphi|_S \colon S \dashrightarrow \P^1$. 
Proposition \ref{vertcomp} states that this does hold for every component if the order of $Q$ is larger than $6$. Moreover, Proposition \ref{vertcomp} is sharp in the sense that there are examples where the order of $Q$ is $6$ and $\C_Q(5)$ has a component that maps under $\sigma$ to $Q$.

\begin{prop}\label{vertcomp}
Suppose the order of $Q$ in $S_0^\ns(k)$ is 
larger than $5$ and $\pC_Q(5)$ has a component $\C_0$ that maps 
under $\sigma \colon \pC_Q(5) \dashrightarrow S$ to a fiber of $\varphi$. Then $Q$ has order $6$ and $\sigma(\C_0)=Q$.
The curve $\pC_Q(5)$ has a unique second component, which is sent under $\sigma$ to a horizontal curve on $S$.
\end{prop}
\begin{proof}
Since $\C_0$ is projective, it contains a point in $\Omega = \pC_Q(5) - \C_Q(5)$, 
say $R$. By Corollary \ref{cor310}, part (4), we have $\varphi(\sigma(R)) = (0:1) \in \P^1$.
Suppose $\sigma$ does 
not send $\C_0$ to a horizontal curve. Then the composition $\varphi \circ \sigma$ sends 
$\C_0$ to $(0:1)$. From Lemma \ref{Qabove0} we find $\sigma(\C_0)=Q$ and we obtain $Q = \sigma(R) = -4Q$ or 
$Q = \sigma(R) = -5Q$ from Proposition \ref{imageinf}.
As the order of $Q$ is larger than $5$, we find that the order is $6$ 
and $\sigma(R) = -5Q$. 
 
We have $c_2^2+4c_1c_5 = \phi_2^2(9\phi_4^2 -4\phi_3\phi_4\psi + 4\phi_3^3)$. From equations \eqref{idPhis} we find 
\begin{equation}\label{simple7}
\phi_3\phi_4\psi-\phi_3^3 = \phi_4(\phi_4+\phi_2^2) - \phi_3^3 = \phi_4^2+\phi_5 = 2\phi_4^2+\phi_6,
\end{equation}
so the factor $9\phi_4^2 -4\phi_3\phi_4\psi + 4\phi_3^3$ equals 
\begin{equation}\label{phi4phi6}
9\phi_4^2 -4\phi_3\phi_4\psi + 4\phi_3^3 = 
9\phi_4^2 -4(2\phi_4^2+\phi_6) =
\phi_4^2-4\phi_6.
\end{equation}
As $\phi_6 = 0$ (together with $y_0 \neq 0$) implies $\phi_4 \neq 0$, we get  
$c_2^2+4c_1c_5 \neq 0$, which in turn, together with $c_1 \neq 0$, implies that $\C_Q(5)$ is reduced. 
Suppose that each component of $\pC_Q(5)$ maps under $\sigma$ to a fiber of $\varphi$. Then as above, 
we find $(\varphi \circ \sigma)(\pC_Q(5)) = (0:1)$ and as the composition 
$\varphi \circ \sigma$ is given by $(-F_5 : F_6)$, we find that 
$F_5$ vanishes on $\C_Q(5)$; as $\C_Q(5)$ is reduced, this implies that
if we view $F_4$ and $F_5$ as polynomials in $k[p,q]$ (cf. Remark \ref{F4F5F6polys}), then 
$F_5$ is a multiple of $F_4$. Viewing $F_4$ and $F_5$ as quadratic polynomials in $q$ over $k[p]$, and 
comparing the coefficients in $k[p]$ of $q^2$ in
\begin{align*}
\phi_2^3 F_4 &= \phi_2^2\phi_3 q^2 + (-3\phi_2\phi_4 p^2 + \ldots)q + \ldots , \\ 
\phi_2^3 F_5 &= \phi_2\big((\phi_2^2-2\phi_4)p - \psi l_1\big) q^2 + 
                \big((\phi_4\psi-4\phi_3^2)p^3 + \ldots \big)q + \ldots, 
\end{align*}
we find 
$$
\phi_2\phi_3 F_5 = \big((\phi_2^2-2\phi_4)p - \psi l_1\big) F_4.
$$
Comparing the coefficient of $p^3q$ in this equality gives
\begin{equation}\label{comprsides}
\phi_3 (\phi_4\psi-4\phi_3^2) = -3\phi_4(\phi_2^2-2\phi_4).
\end{equation}
Since $\phi_4-\psi\phi_3+\phi_2^2=0$ by equations \eqref{idPhis},
we find from \eqref{phi4phi6} that the difference of the two sides in \eqref{comprsides} equals 
$$
-3\phi_4(\phi_2^2-2\phi_4) - \phi_3 (\phi_4\psi-4\phi_3^2) + 3\phi_4(\phi_4-\psi\phi_3+\phi_2^2) = 9\phi_4^2 -4\phi_3\phi_4\psi + 4\phi_3^3 = \phi_4^2 - 4\phi_6.
$$
Hence, the equality \eqref{comprsides} is equivalent to $4\phi_6 = \phi_4^2$, so
we obtain $\phi_4 = \phi_6 =0$, a contradiction from which we conclude that not all 
components map to a vertical component. It follows that there is a second component, which is 
unique as $c_1 \neq 0$ implies that there are at most two components. This second component maps to 
a horizontal curve on $S$. 
\end{proof}

We say that two pairs $(X_1,Q_1)$ and $(X_2,Q_2)$ of a variety with a point on it are
isomorphic if there is an isomorphism from $X_1$ to $X_2$ that maps $Q_1$ to $Q_2$. 
For example, the involution $\iota \colon \P \to \P$ that sends 
$(x:y:z:w)\in \P$ to $(x:y:-z:w-z)$ fixes $Q$, so it induces an isomorphism, 
also denoted~$\iota$, from the pair $(S,Q)$ to $(\iota(S),Q)$;
the surface $\iota(S)$ is given by $y^2 = x^3+\tilde{f}(z,w)x +\tilde{g}(z,w)$, with 
\begin{align*}
\tilde{f}(z,w) &= f(-z,w-z) = f_0w^4 +(-4f_0-f_1) w^3z + \dots, \\
\tilde{g}(z,w) &= g(-z,w-z) = g_0w^6 +(-6g_0-g_1) w^5z + \dots.
\end{align*}
Note that $\iota$ fixes the points in the fiber above $(0:1)$ and it 
switches the fibers above $(1:1)$ 
and $(1:0)$. It also fixes $f_0$ and $g_0$ and it replaces $f_1$ and $g_1$ by 
$-4f_0-f_1$ and $-6g_0-g_1$, respectively. 

The following lemma is well known (see \cite[Proposition 8.2.8.]{CE} for $n=5$,  \cite[pp. 457]{Igusa_FS} for $n=3$, and \cite[Table 3]{Kub} for characteristic $0$). The lemma is used in Propositions \ref{outofhell5Q} and \ref{nohell3Q}.

\begin{lemma}\label{explicitunivcurve}
Let $E$ be an elliptic curve over $k$ and $n \in \{3,5\}$ an integer.  Let $P\in E(k)$ be a point of order $n$. 
Then there exist elements $\beta \in k$ and $e \in \{0,1\}$ such that the pair $(E,P)$ is isomorphic to the pair 
$(E',(0,0))$, with $E'$ given by 
$$
\begin{cases}
y^2+exy+\beta y = x^3 & \mbox{if } n=3, \\
y^2+(\beta+1)xy+\beta y = x^3+\beta x^2 & \mbox{if } n=5.
\end{cases} 
$$
\end{lemma}
\begin{proof}
After choosing an initial Weierstrass model for $E$, we may apply a linear change of variables to obtain a model $E'$ in which $P$ corresponds to $(0,0)$. Given that the order of $P$ is not $2$, we may also assume that the tangent line to the model $E'$ at $P$ is given by $y=0$. Then there are $a_1,a_2,a_3 \in k$ with $a_3 \neq 0$ such that $E'$ is given by $y^2+a_1xy+a_3y=x^3+a_2x^2$. 
We have $-2P = (-a_2,0)$, so $3P=0$, or, equivalently, $P = -2P$, holds if and only if $a_2=0$. 

If $n=3$, so $3P=0$, then either $a_1=0$ or $a_1\neq 0$, and in the latter case, we may scale $x$ and $y$ such that we have $a_1=1$. These two cases are exactly the claimed cases, with $\beta=a_3$ and $e=a_1$. 

If $n=5$, then we have $a_2\neq 0$ and $a_3\neq 0$, so we may scale $x$ and $y$ such that we have $a_2=a_3$. 
Then we have $3P=(-a_1 + 1 , a_1 - a_2 - 1 )$, so 
the property $5P=0$, or, equivalently, $3P = -2P$, yields $a_1=a_2+1$, which yields the claimed case with $\beta=a_2$. 
\end{proof}

\begin{prop}\label{outofhell5Q}
Suppose that the characteristic of $k$ is not $5$, and $5Q = \O$ in $S_0^\ns(k)$.
If no component of $\pC_Q(5)$ maps under $\sigma$ to a horizontal curve on $S$, then 
there exist $\beta, \delta \in k$ such that the pair $(S,Q)$ is isomorphic to the 
pair of Example \ref{ex:hell5Q}.
\end{prop}
\begin{proof}
If $E$ is an elliptic curve over $k$ with a point~$P$ of order~$5$, then by Lemma \ref{explicitunivcurve}, there 
exists a $\beta \in k$ such that $E$ is isomorphic to the 
elliptic curve given by $y^2+(\beta+1)xy+\beta y = x^3+\beta x^2$, with $P$ corresponding 
to the point $(0,0)$. 
A short Weierstrass model for this curve is given by $v^2 = u^3+Au+B$, 
with the point $(0,0)$ corresponding to $(u_0,v_0)$, where  
\begin{align*}
u_0&=3(\beta^2+6\beta+1),\\
v_0&=108\beta,\\
A&=-27(\beta^4+12\beta^3+14\beta^2-12\beta+1),\\
B&=54(\beta^2+1)(\beta^4+18\beta^3+74\beta^2-18\beta+1).
\end{align*}
If $S_0$ is smooth, then, as $Q$ has order $5$ on $S_0$ and isomorphisms between 
short Weierstrass models are all given by appropriate scaling of the coordinates, 
there are $\beta, \eta \in k$ such that 
\begin{equation}\label{param5}
(x_0,y_0,f_0,g_0)=(u_0\eta^2, v_0\eta^3, A\eta^4, B\eta^6).
\end{equation}
Another way to phrase this is that \eqref{param5} gives a parametrization of the quadruples $(x_0,y_0,f_0,g_0)$ with $y_0^2 = x_0^3+f_0x_0+g_0$ for which the associated fifth division polynomial 
$\Phi_5 \in k[f_0,g_0][x]$ vanishes at $x_0$. Hence, also in the case that $S_0$ is singular, there exist $\beta,\eta \in k$ for which \eqref{param5} holds. From $y_0 \neq 0$, we get $\beta,\eta \neq 0$.
Without loss of generality, we assume $\eta = 1$.
The fiber $S_0$ is singular if and only if $D = \beta(\beta^2+11\beta-1)$ is zero, and in this case $S_0$ is nodal. Note that because $Q$ has order $5$, we have $\phi_5 = 0$ and $\phi_3, c_1 \neq 0$. 

We first state two claims, both with a computational proof. 

\noindent {\bf Claim 1}: {\em  If $D=0$ and $F_4$ divides $F_5F_6$, then $S$ is singular.}

\begin{computations}
{\em Proof}. Since the main coefficient $c_1$ of $F_4$ as a polynomial 
in $q$ over $k[p]$ is invertible, we can compute (by computer, with $x_0, f_0, \ldots, f_4, g_0, \ldots, g_6$ independent transcendentals) the remainder of $F_5F_6$ upon division by $F_4$, which is a polynomial $L = \mu q +\nu$, with $\mu,\nu \in k[p]$
of degree $9$ and $11$, respectively. Our special values of $x_0,f_0,g_0$ already imply that the coefficients of $p^{11}$ and $p^9q$ in $L$ specialize to $0$, and the fact that $F_4$ divides $F_5F_6$ implies that $L$ specializes to $0$. We consider two cases, based on the characteristic of $k$. 

\noindent {\bf Case 1}: the characteristic of $k$ is not $11,17,23$, or $29$. 

In this case, the vanishing of the (specialization of the) coefficients of $p^8q$, $p^7q$, $p^6q$, and $p^5q$ in $L$ determine, in that order, the values of $f_1, f_2, f_3$, and $f_4$ in terms of $g_0, \ldots, g_6$. The vanishing of the coefficient of $p^8$ then implies that we have one of two subcases: (a) $g_1=0$ or (b) $g_2 = \lambda g_1^2$ for some specific constant $\lambda$. 

Assume we are in subcase (a), i.e.,  $g_1=0$. The vanishing of the coefficient of $p^7$ yields $g_2=0$; then the vanishing of the coefficients of $p^5$ and $p^3q$ implies $g_3=g_6=0$, and finally the vanishing of the coefficients of $p^3$ gives $g_4=0$, which shows that the pair $(S,Q)$ is isomorphic to the pair in Example \ref{ex:hell5Q}, with $\delta = g_5$, though $F_6$ vanishes on both components $\C_0$ and $\C_1$, and $S$ is singular.

Assume we are in subcase (b), i.e.,  $g_2 = \lambda g_1^2$. We may assume $g_1\neq 0$, and the vanishing of the coefficients of $p^7$, $p^6$, $p^5$, and finally $p^3q$, express $g_3, g_4, g_6, g_5$, in that order, in terms of the remaining unknown coefficients of $g$, which in the end yields a surface $S$ that is singular. \par

\noindent {\bf Case 2}: the characteristic of $k$ is not $7,13$, or $19$. 

As in case 1, we similarly solve for the parameters $f_1, \ldots, f_4$ and $g_1, \ldots, g_6$, except that we start by expressing $g_1,\ldots, g_4$ in terms of $f_1, \ldots, f_4$. We conclude also in these characteristics that $S$ is singular, thus proving the claim. 
\end{computations}

\noindent {\bf Claim 2}: {\em If $D \neq 0$ and $\C_Q(5)$ is reduced and $F_4$ divides $(F_5+F_6)F_6$, then either $S$ is singular, or there exists a $\delta \in k$, such that 
the pair $(S,Q)$ is isomorphic to the pair of Example \ref{ex:hell5Q}. }

\begin{computations}
{\em Proof}. Since $\C_Q(5)$ is reduced, i.e., $F_4$ has no multiple factors, the condition that 
$F_4$ divides $(F_5+F_6)F_6$ is equivalent to all components of $\C_Q(5)$ being sent under 
the map $\varphi \circ \sigma$ to $(1:1)$ or $(1:0)$. As the isomorphism $\iota$ described before 
Lemma \ref{explicitunivcurve} switches the fibers above these two points, 
the hypotheses of this claim hold for 
the pair $(S,Q)$ if and only if they hold for the pair $(\iota(S),Q)$. 
Hence, without loss of generality we may apply $\iota$ at some point. 

Viewing $F_4,F_5,F_6$ as polynomials in $q$ 
over $k[p]$, we find that generically, say over a field in which $x_0,f_0, \ldots, f_4, g_0, \ldots, g_6$ are independent transcendentals, there are 
$d_0, \ldots ,d_{10}$ and $e_0, \ldots, e_{12}$, in terms of these transcendentals, such that 
$$
(F_5+F_6)F_6 \equiv  (d_{10}p^{10}+\dots + d_1p+d_0)q+ e_{12}p^{12}+\dots + e_1p+e_0 \quad \pmod {F_4}.
$$ 
The fact that $Q$ has order $5$ implies that $d_{10},d_{9},e_{12},e_{11}$ specialize to $0$. 
In our case, the other coefficients $d_0, \ldots ,d_{8}, e_0, \ldots, e_{10}$ specialize to $0$ as well. 
We claim that from the fact that $e_{10}$ and $d_8$ specialize to $0$, it follows that 
\begin{equation}\label{pairs}
\left\{
\begin{array}{l}
f_1 =0\\
g_1 = 0
\end{array}
\right.\qquad \mbox{or}\qquad
\left\{
\begin{array}{l}
f_1 = -4f_0\\  
g_1 =-6g_0
\end{array}
\right. \qquad \mbox{or} \qquad 
\left\{
\begin{array}{l}
f_1 = - 2f_0 - 54\gamma^{-1}\lambda \\
g_1 = -3g_0 + 54\gamma^{-1}\mu
\end{array}
\right.
\end{equation}
for some element $\gamma \in k$ with $\gamma^2 = 5$ and with 
\begin{align*}
\lambda=&\,\,(\beta^2+1)(\beta^2+10\beta-1),\\
\mu=&\,\,3(\beta^6 + 16\beta^5 + 49\beta^4 - 40\beta^3 - 49\beta^2 + 16\beta - 1).
\end{align*}
Indeed, for any $\gamma$ in an extension of $k$ with $\gamma^2=5$ and $\omega = \frac{1}{2}(3-\gamma)$, 
 the linear combinations
\begin{align*}
\frac{1}{2^5 3^3}& \phi_2^4\left((3\beta-1)(7\beta+1)d_8 + 180\beta(11\beta-2)e_{10}\right) \\
 \quad & = \big(3g_0f_1-2f_0g_1\big)\cdot\big((f_1+2f_0)\mu+(g_1+3g_0)\lambda\big)\quad \mbox{and}\\
\frac{1}{4}&\phi_2^4\left( \omega^4\beta^2 + (2\gamma-4)\beta + \omega^{-1} \right) (d_8 + 36\omega e_{10}) \\
\quad &=  \big( \gamma(3g_0f_1-2f_0g_1) - 54(g_1\lambda+f_1\mu) \big) \\
  & \qquad \cdot \big( \gamma(3g_0f_1-2f_0g_1) - 54((g_1+6g_0)\lambda+(f_1+4f_0)\mu) \big)
\end{align*}
of $d_8$ and $e_{10}$ factor into two linear factors. Therefore, the vanishing of $d_8$ and $e_{10}$ 
implies the vanishing of one of the first two factors and one of the second two. The four combinations 
give four systems of two linear equations in the two variables $f_1$ and $g_1$. For each combination, the 
determinant of the system is a nonzero multiple of $D$ and therefore nonzero itself. The systems 
yield exactly the four claimed pairs for $(f_1,g_1)$ in \eqref{pairs}. 

Note that as the isomorphism $\iota$ replaces $f_1$ and $g_1$ by $-4f_0-f_1$ and 
$-6g_0-g_1$, respectively, it switches the first two cases in \eqref{pairs}, 
as well as the last two cases given by the third pair for $\pm \gamma$. Therefore, 
after applying the isomorphism $\iota$ if necessary, we may assume we have only two 
subcases. 

\noindent {\bf Case 1}: We have $(f_1,g_1) = (0,0)$.

The equations $d_7=e_9=0$ determine a system of two linear equations in $f_2$ and $g_2$, 
of which the determinant is a nonzero multiple of $D$ and therefore nonzero itself. 
The unique solution is $f_2=g_2=0$. Subsequently, the system $d_6=e_8=0$ gives 
$f_3=g_3=0$ and then the system $d_5=e_7=0$ yields $f_4=g_4=0$. At this point, 
the coefficients $d_4$ and $e_6$ specialize to $0$ automatically, and the equation $d_3=0$ 
determines $g_6=0$. With $g_5=\delta$, we obtain exactly the surface of 
Example \ref{ex:hell5Q}. 

\noindent {\bf Case 2}: We have 
$(f_1,g_1) =  (- 2f_0 - 54\gamma^{-1}\lambda, -3g_0 + 54\gamma^{-1}\mu)$. 

As in the previous subcase, the linear systems $d_{9-i}=e_{11-i}=0$ determine 
$f_i$ and $g_i$ inductively for $i=2,3,4$. Again, the 
coefficients $d_4$ and $e_6$ then specialize to $0$ automatically. Finally, the 
system $d_3=e_6=0$ is linear in $g_5$ and $g_6$ and determines these
two parameters uniquely. However, this yields a surface $S$ that is singular. 
More specifically, 
the associated minimal elliptic surface has two singular fibers of type $I_5$.
This proves the claim. 
\end{computations}

We continue the proof of the proposition. 
Suppose no component of $\pC_Q(5)$ maps under $\sigma$ to a horizontal curve on $S$, so $\varphi \circ \sigma$ has finite image.
If we had $\varphi(\sigma(\pC_Q(5)) = (0:1)$, then we would have $\sigma(\pC_Q(5)) = Q=-4Q$ by Lemma \ref{Qabove0}, so by Corollary \ref{cor310}, part (3), the fiber $S_0$ would be cuspidal. From this contradiction we conclude that there is a component $\C_1$ with $\varphi(\sigma(\C_1)) \neq (0:1)$. 
Without loss of generality, we assume $\varphi(\sigma(\C_1)) = (1:0) =:\infty$ and we write $S_\infty = \varphi^{-1}(\infty)$.

We will distinguish the following three cases. 
\begin{itemize}
\item[(A)] There is a component $\C_0$ of $\pC_Q(5)$ with $\varphi(\sigma(\C_0)) =(0:1)$.
\item[(B)] There is a component $\C_0$ of $\pC_Q(5)$ with $\varphi(\sigma(\C_0)) \neq (0:1),(1:0)$.
\item[(C)] There is no component $\C_0$ of $\pC_Q(5)$ with $\varphi(\sigma(\C_0)) \neq (1:0)$.
\end{itemize}
Since $c_1 \neq 0$, the curve $\pC_Q(5)$ has at most two components and both are reduced if there are two, so in cases (A) and (B), the components are $\C_0$ and $\C_1$, and $\C_Q(5)$ is reduced.

We start with case (A).
Assume that there is a component $\C_0$ of $\pC_Q(5)$ with $\varphi(\sigma(\C_0)) = (0:1)$.  From Lemma \ref{Qabove0} we find $\sigma(\C_0) = Q =-4Q$, and as $\C_0$ contains points of $\Omega$, we conclude that $S_0$ is singular from Corollary \ref{cor310}, part (1), so $\beta^2+11\beta-1=0$. 
If we consider $F_4, F_5$, and $F_6$ as polynomials in $q$ over $k[p]$ (cf.~Remark~\ref{F4F5F6polys}), then $F_5$ and $F_6$ vanish on $\C_0$ and $\C_1$, respectively, so $F_4$ divides $F_5F_6$. Claim 1 implies that $S$ is singular, a contradiction.

We continue with case (B).
Assume that there is a component $\C_0$ of $\pC_Q(5)$ with $\varphi(\sigma(\C_0)) \neq (0:1), (1:0)$. After applying an automorphism of the 
base curve $\P^1$ that fixes $(0:1)$ and $(1:0)$, we may assume $\varphi(\sigma(\C_0)) = (1:1)$, 
so that $F_5+F_6$ and $F_6$ vanish on $\C_0$ and $\C_1$, respectively, and the product 
$(F_5+F_6)F_6$ is divisible by $F_4$. 
Since the points in $\Omega = \pC_Q(5) - \C_Q(5)$ 
map under $\sigma$ to $S_0$, and $\varphi(S_0-\{\O\})=(0:1)$, the points in $\Omega$ map under $\sigma$ to $\O=-5Q$; it follows from Corollary \ref{cor310}, part (1), that $S_0$ is smooth, 
so we find $D \neq 0$. Hence, we are done by Claim 2. 

We finish with case (C). In that case, we have $\sigma(\pC_Q(5)) \subset S_\infty$, 
so $F_6$ vanishes on $\C_Q(5)$, and also $\sigma(\Omega) \subset S_\infty$. 
From Proposition \ref{imageinf} we conclude $\sigma(\Omega) \subset S_0 \cap 
S_\infty = \{\O\} = \{-5Q\}$; it follows from Corollary \ref{cor310}, part (1), that 
$S_0$ is smooth, so we find $D \neq 0$. 
From \eqref{simple7} we obtain 
$$
c_2^2+4c_1c_5 = \phi_2^2(9\phi_4^2 -4\phi_3\phi_4\psi + 4\phi_3^3)= 
  \phi_2^2(9\phi_4^2 -4(\phi_4^2+\phi_5)) = \phi_2^2(5\phi_4^2 -4\phi_5).
$$
As $\phi_5 = 0$ (together with $y_0 \neq 0$) implies $\phi_4 \neq 0$, we get  
$c_2^2+4c_1c_5 \neq 0$, which in turn, together with $c_1 \neq 0$, implies that $\C_Q(5)$ is reduced.
Therefore, $F_6$ is a multiple of $F_4$, so we are done by claim~2. 
\end{proof}

Indeed, in characteristic $5$, there are other examples than those mentioned in Proposition \ref{outofhell5Q} where $Q$ has order $5$ and no component of $\pC_Q(5)$ maps under $\sigma$ to a horizontal curve on $S$ (see Example \ref{ex:char5}). 
It takes less computational force to deal with the case that $Q$ has order $4$.

\begin{prop}\label{nohell4Q}
Suppose $4Q=\O$ in $S_0^\ns(k)$. Then $\C_Q(5)$ has a component that maps under $\sigma \colon \C_Q(5) \dashrightarrow S$ to a horizontal curve on $S$.
\end{prop}
\begin{proof}
First note that the fiber $S_0$ does not have a cusp, as the additive reduction together with the identity $4Q=\O$ would imply that the characteristic of $k$ is $2$, which it is not by assumption. Therefore, by Corollary \ref{cor310}, parts (1) and (2), at least one of the points in $\Omega$ maps to $-5Q=-Q$. 
Let $R$ be such a point and let $\C_0$ be a component of $\pC_Q(5)$ that contains $R$. 
Suppose that $\C_0$ is sent by $\sigma$ to a fiber on $S$, so that $\varphi(\sigma(\C_0))$ is a point on $\P^1$. 
From $\sigma(R) = -Q \in S_0$ we conclude $\varphi(\sigma(\C_0)) = (0:1)$, and Lemma \ref{Qabove0} implies $\sigma(\C_0) = Q$, 
which contradicts $\sigma(R) =-Q$, so $\C_0$ is sent to a horizontal curve on~$S$. 
\end{proof}

Finally, we deal with the case that $Q$ has order $3$. 

\begin{prop}\label{nohell3Q}
Suppose $3Q=\O$ in $S_0^\ns(k)$. Then $\C_Q(5) \subset \A^2(p,q)$ has a unique component that projects birationally to $\A^1(p)$. If this component maps under  
$\sigma \colon \C_Q(5) \dashrightarrow S$ to a vertical curve on $S$, then the pair $(S,Q)$ is isomorphic to one of the pairs described in Examples \ref{ex:order3} and \ref{ex:order3sec}.
Moreover, if another component of $\C_Q(5)$ maps to a vertical curve on $S$ as well, then the pair $(S,Q)$ is isomorphic to one of the pairs described in Subexamples \ref{ex:order3}(iii) and  \ref{ex:order3sec}(iii).
\end{prop}
\begin{proof}
If $E$ is an elliptic curve over $k$ with a point $P \in E(k)$ of order $3$, then by Lemma \ref{explicitunivcurve}, there exist elements $\beta' \in k$ and $e \in \{0,1\}$ such that $E$ is isomorphic to the elliptic curve given by $y^2+exy+\beta' y = x^3$, with the point $P$ corresponding to $(0,0)$. 
Associated short Weierstrass models are 
given by $v^2 = u^3 + A_eu+B_e$, with the point $(0,0)$ corresponding to $(u_e,v_e)$, where 
\begin{align*}
u_e&=3e,\\
v_e &= \beta,\\
A_e & = (6\beta-27)e,\\
B_e &= \beta^2-18(\beta-3)e,
\end{align*}
and where $\beta = 108\beta'$. We also simplified these expressions using $e^2=e$. 
Since $Q$ has order $3$ in $S^\ns(k)$, we find, as in the proof of Proposition \ref{outofhell5Q}, 
that there are $\beta, \eta \in k^*$ and $e\in \{0,1\}$, such that 
$$
(x_0,y_0,f_0,g_0) = (u_e\eta^2,v_e\eta^3,A_e\eta^4,B_e\eta^6),
$$
also in the case that $S_0$ is singular; the property $\beta \neq 0$ follows from the assumption $y_0 \neq 0$. 
Without loss of generality we assume $\eta=1$.

Any component $\C_0$ of $\pC_Q(5)$ contains a point $R \in \Omega = \pC_Q(5) - \C_Q(5)$, which satisfies $\varphi(\sigma(R)) = (0:1)$ by Corollary \ref{cor310}, part (4), so if $\sigma(\C_0)$ is contained in a fiber of $S \dashrightarrow \P^1$, then $F_5$ vanishes on $\C_0$. 

As the order of $Q$ is $3$, we have $\phi_3=0$ and thus $\phi_2\phi_4 \neq 0$, so $c_1 = 0$ and $c_2 \neq 0$, and 
the curve $\C_Q(5)$ is given by $mq=n$ with 
$$
m = c_2p^2+c_3p+c_4 \qquad \mbox{and} \qquad n = c_5p^4+ c_6p^3+c_7p^2+c_8p+c_9. 
$$
From $c_2 \neq 0$, we find that $m$ is not identically $0$, so indeed, 
there is a unique  component of $\C_Q(5)$, say $\C_1$, 
that projects birationally to $\A^1(p)$. Assume that $\sigma(\C_1)$ is contained in a fiber on $S$. Then $F_5$ vanishes
on $\C_1$ by the above. If we write 
$F_5$ as $F_5 = \delta_1 q^2+ \delta_2 q +\delta_3$, with $\delta_j \in k[p]$ of degree $2j-1$, then we find that  
$$
L = \delta_1 n^2 + \delta_2mn + \delta_3 m^2 
$$
vanishes. 

In the case $e=1$, so $(x_0,y_0,f_0,g_0) = (u_1,v_1,A_1,B_1)$, we claim that the pair $(S,Q)$ is isomorphic to one of the pairs of Example \ref{ex:order3}. We sketch a sequence of computations that proves the claim. 

\begin{computations}
A priori, say over a field in which the elements 
$x_0, f_0, \ldots, f_4, g_0, \ldots, g_6$ are independent transcendentals, 
the polynomial $L\in k[p]$ has degree $9$, but from $\phi_3=0$, it already follows that the 
degree is at most $8$. We will use the vanishing of all coefficients to identify the pair $(S,Q)$. 

The vanishing of the coefficient of $p^8$ in $L$ gives $3f_1g_0 = 2f_0g_1$. Since $f_0$ and $g_0$ 
do not both vanish, there is a $\delta \in k$ such that 
$f_1 = 2\delta f_0$ and $g_1 = 3\delta g_0$. After applying an automorphism of $\P^1$ given by 
$(z:w) \mapsto (2z: \delta z +2w)$, we may assume without loss of generality that $\delta=0$, so 
$f_1=g_1=0$. Then the vanishing of the coefficient of $p^7$ in $L$
shows that there is an $\alpha_1 \in k$ such that $f_2 = (18-3\beta)\alpha_1$ and 
$g_2 = (15\beta-54)\alpha_1$.
The coefficient of $p^6$ now vanishes automatically and the vanishing of the coefficient of $p^5$ yields
$g_4=(\beta - 3)(f_4 - 3\alpha_1^2)$. Subsequently, the vanishing of the coefficient of $p^4$ gives 
$g_5= \big((2\beta - 9)f_3- 3g_3\big)\alpha_1\beta^{-1}$. Then the coefficient of $p^3$ vanishes automatically. 
The vanishing of the coefficient of $p^2$ yields 
$f_4= - 3\alpha_1^2$ or $3g_3 = (\beta - 9)f_3$, but in the latter case, the vanishing of the coefficient 
of $p$ yields the former, so we have $f_4=-3\alpha_1^2$ in any case.  
Finally, the vanishing of the coefficient of $p$ gives $3g_3 = (\beta - 9)f_3$ or $\alpha_1=0$; 
the former case yields Example \ref{ex:order3}(i) with $\alpha_2=\frac{1}{3}f_3$ and $\alpha_3=g_6$, 
while the latter case yields Example \ref{ex:order3}(ii) with $\alpha_4=\frac{1}{3}f_3$, 
$\alpha_5 = g_3$ and $\alpha_6 = g_6$. This proves the claim.
\end{computations}

We continue (still) with $e=1$. 
Suppose we are in the case of Example \ref{ex:order3}(i). If $\sigma$ sends one of the components of $\C_Q(5)$ given by $p^2-\beta\alpha_1=0$ to a fiber of $\varphi$, then the first argument of this proof shows that 
$\frac{1}{3}\beta F_5=p(q+\alpha_1)(\beta q - p^2 + 2\beta\alpha_1)$ vanishes on this component, which implies $\alpha_1=0$, 
so the pair $(S,Q)$ belongs to the family described in Example \ref{ex:order3}(iii). 
Now suppose we are in the case of Example \ref{ex:order3}(ii). If $\sigma$ sends the component of $\C_Q(5)$ given by $p=0$ to a fiber of $\varphi$, then similarly $F_5=3\beta^{-1}q(\beta pq - p^3 + (\beta - 9)\alpha_4 - \alpha_5)$ vanishes on this component, which implies $\alpha_5 = (\beta-9)\alpha_4$, so again the pair $(S,Q)$ belongs to the family described in Example \ref{ex:order3}(iii). This finishes the case $e=1$.

We now consider the case $e=0$, so $(x_0,y_0,f_0,g_0) = (0,\beta,0,\beta^2)$. 
We claim that the pair $(S,Q)$ is isomorphic to one of the pairs of Example \ref{ex:order3sec}. We sketch a sequence of computations that proves the claim. 

\begin{computations}
Since $g_0 \neq 0$, we may apply an automorphism of $\P^1(z,w)$ given by $(z:w) \mapsto (6g_0z: g_1 z + 6g_0w)$ to reduce to the case $g_1=0$.
As in the case $e=1$, we will use the vanishing of all coefficients in $L$ 
to identify the pair $(S,Q)$. 
The vanishing of the coefficients of $p^8$, $p^7$, $p^5$, and $p^4$ yields $f_1=g_2=f_4=g_5=0$. 
Then the vanishing of the coefficient of $p^2$ yields $f_3g_4=0$; if $f_3=0$, then the vanishing of the coefficient of $p$ in $L$ gives $g_4=0$, so we have $g_4=0$ in any case. Then the vanishing of the coefficient of $p$ gives $f_2=0$ or $f_3=0$ and these cases correspond to Subexamples \ref{ex:order3sec}(ii) and \ref{ex:order3sec}(i), respectively. This proves the claim.
\end{computations}

Suppose we are in the case of Example \ref{ex:order3sec}(i), so $f_3=0$ and $f_2=\alpha$. If $\sigma$ sends one of the components of $\C_Q(5)$ given by $3p^2+\alpha=0$ to a fiber of $\varphi$, then the first argument of this proof shows that 
$F_5=3pq^2$ vanishes on this component, which implies $\alpha=0$, 
so the pair $(S,Q)$ belongs to the family described in Example \ref{ex:order3sec}(iii). 
Now suppose we are in the case of Example \ref{ex:order3sec}(ii), so $f_2=0$ and $f_3=\alpha$. If $\sigma$ sends the component of $\C_Q(5)$ given by $p=0$ to a fiber of $\varphi$, then similarly $F_5=q(3pq+\alpha)$ vanishes on this component, which implies $\alpha=0$, so again the pair $(S,Q)$ belongs to the family described in Example \ref{ex:order3sec}(iii). This
finishes the case $e=0$ and thus the proof.
\end{proof}

\begin{cor}\label{alwayshorizontal}
Let $k$, $S$, and $Q$ be as before. Assume that the pair $(S,Q)$ is not isomorphic to the one of the pairs described in Examples \ref{ex:hell5Q}, \ref{ex:order3}(iii), and \ref{ex:order3sec}(iii). 
If the order of $Q$ in $S_0^\ns(k)$ is $5$, then also assume that the characteristic of $k$ is not equal to $5$. 
Then the rational map $\sigma \colon \C_Q(5) \dashrightarrow S$ sends at least one component of 
$\C_Q(5)$ to a horizontal curve on $S$. 
\end{cor}
\begin{proof}
From the assumption $y_0\neq 0$, it follows that the order of $Q$ is at least $3$. 
The statement now follows immediately from Propositions \ref{vertcomp}, \ref{outofhell5Q}, \ref{nohell4Q}, and \ref{nohell3Q}.
\end{proof}

\begin{remark}\label{whatisthiscurve}
In Remark \ref{Tremark}, we have seen that the closure of the image $\sigma(\pC_Q(5))$ in $S$ is contained in the intersection $S \cap T$, where $T$ is the image of $\gamma\colon \Gamma \to \P$. 
Generically, this containment is in fact an equality, but if there are any $(-1)$-curves of $S$ going through $Q$, then these are components of the intersection $S \cap T$ as well. In degenerate cases, the intersection $S \cap T$ may contain even more components (see Remark \ref{finalremarkchapter5}). When studying our del Pezzo surfaces of degree one in families, it is more natural to look at the divisor $S \cap T$ on $S$ than at the closure of the image $\sigma(\pC_Q(5))$.

We will now describe this intersection $S \cap T$ and its arithmetic genus in terms of the Picard group 
of $S$, at least in the 
generic case. Generically, the Picard group $\Pic S$ of $S$ is generated by the divisor class of $-K_S$. 
Also generically, the surface $T$ has degree $12$ and the intersection $S \cap T$ is irreducible and reduced, so 
$\sigma(\pC_Q(5)) = S \cap T$ is linearly equivalent with $-12K_S$. The arithmetic genus of $\sigma(\pC_Q(5))$ is $67$ in this case, and $\sigma(\pC_Q(5))$ has multiplicity $10$ at the point $Q$, multiplicity $2$ at $-5Q$, and is also singular at $20$ more points, which agrees with the fact that the geometric genus of the normalization equals 
$$
67 - \tfrac{1}{2}\cdot 10 \cdot (10-1) - \tfrac{1}{2} \cdot 2 \cdot (2-1) - 20(\tfrac{1}{2} \cdot 2 \cdot (2-1)) = 1.
$$
These $22$ singular points of $\sigma(\pC_Q(5))$ are the intersection points of $S$ with the singular locus of $T$, which is a curve with an embedded point at $Q$. 

Conversely, the family of intersections of $S$ with a hypersurface of degree $12$ has dimension 
$78$, as can be seen from the fact that the space of polynomials in $x,y,z,w$ of weighted degree $12$ modulo the multiples of the defining equation of $S$ has dimension $102 - 23= 79$ or from the fact that the linear system of curves in $\P^2$ of degree $3d$ having multiplicity at least $d$ at each of $8$ given points has dimension $\binom{3d+2}{2} - 1 - 8\cdot \binom{d+1}{2} = \binom{d+1}{2}$, which equals $78$ for $d=12$. 
Hence, the subfamily of those intersections that have multiplicity $10$ at $Q$, multiplicity $2$ at $-5Q$, and $20$ more singularities, has dimension
$$
78 - \tfrac{1}{2} \cdot 10 \cdot(10+1) - \tfrac{1}{2} \cdot 2 \cdot (2+1) - 20 = 0,
$$
so there are only finitely many curves satisfying these conditions. 

We now give yet another description of $\sigma(\pC_Q(5))$ that narrows it down to one of 
only finitely many curves. 
Note that the projection $\nu \colon S \to \P(2,1,1)$ from $S$ to the weighted projective space with coordinates $x,z,w$, gives $S$ the structure of a double cover of a cone that is ramified at the singular point $\O$ (corresponding to the vertex of the cone), as well as over the curve given by $x^3+fx+g=0$, i.e., the locus of nontrivial $2$-torsion points. The involution induced by this double cover is multiplication by $-1$ on the elliptic fibration. 
If we let $-\sigma(\pC_Q(5))\subset S$ denote the image of $\sigma(\pC_Q(5))$ under this involution, then $\sigma(\pC_Q(5))$ and $-\sigma(\pC_Q(5))$ intersect each other in $36$ points on the ramification locus of $\nu$, as well as $108$ points off the ramification locus. The image $\nu\big(\sigma(\pC_Q(5))\big) = \nu\big(\!\!-\!\sigma(\pC_Q(5))\big) \subset \P(2,1,1)$ is a curve of degree $24$ that intersects the branch locus of $\nu$ at $36$ points, being tangent at each, that has multiplicity $10$ at $\nu(Q) = (x_0:0:1)$, multiplicity $2$ at $\nu(-5Q)$, and that is singular at $74$ more points (namely the $20$ images under $\nu$ of the remaining singular points of $\sigma(\pC_Q(5))$, and the $54$ images of the intersection points of $\sigma(\pC_Q(5))$ and $-\sigma(\pC_Q(5))$). These properties narrow down the $168$-dimensional family of curves in $\P(2,1,1)$ of degree $24$ to only finitely many curves.
\end{remark}

\begin{remark}\label{finalremarkchapter5}
Given that for all pairs $(S,Q)$ described in the examples in the previous section there are at least six $(-1)$-curves on $S$ going through $Q$, whenever we want to exclude any of these examples, it suffices to assume that $Q$ does not lie on six $(-1)$-curves. 

We will now explain in terms of the image $T$ of $\gamma\colon \Gamma \to \P$ (cf. Remarks \ref{Tremark} and \ref{whatisthiscurve}) why it is not surprising that the existence of many $(-1)$-curves on $S$ through $Q$ is related to the existence of a component of $\pC_Q(5)$ that maps under $\sigma \colon \pC_Q(5) \dashrightarrow S$ to a fiber of $\varphi \colon S \dashrightarrow \P^1$. 

In Example \ref{ex:hell5Q}, the scheme-theoretic intersection $D' = T\cap S$ consists of the 
ten $(-1)$-curves going through $Q$ and the cuspidal fiber $S_\infty$ with multiplicity $2$. As a divisor on $S$,  
we have that $D'$ is linearly equivalent to $-12K_S$ (cf. Remark \ref{whatisthiscurve}).

In general, the pull-back of $S \cap T$ under the blow-up $\E \to S$ is a divisor $D$ on~$\E$ that consists of
\begin{itemize}
\item[(i)] the components of the strict transform $D_0$ of the image $\sigma(\pC_Q(5))$, 
\item[(ii)] the strict transforms of the $(-1)$-curves on $S$ through $Q$,
\item[(iii)] the fiber $\E_0$, and 
\item[(iv)] the zero section $\O$
\end{itemize}
with certain multiplicities. The degree of the restriction of $\pi \colon \E \to \P^1$ to $D$ equals the intersection number of $D$ with any fiber of $\pi$. Generically, this degree equals $\deg T=12$, the multiplicities of $\E_0$ and $\O$ are $0$, and the multiplicities of the components in (i) and (ii) are $1$; the $(-1)$-curves on $S$ intersect fibers with multiplicity $1$, so the restriction of $\pi$ to $D_0$ has degree $12-s$, where $s$ is the number of $(-1)$-curves on $S$ through $Q$. Note that Example \ref{ex:hell5Q} is not generic in the sense that the multiplicity of the zero section $\O$ is $2$, thus reducing the degree of the restriction of $\pi$ to $D_0$ to $0$.

In general, the multiplicities of the components in (iii) and (iv) seem to depend only on the order of $Q$ in $\E_0^\ns(k)$ and the singularity type of $\E_0$, but in any case we find that the more $(-1)$-curves there are on $S$ that go through $Q$, the smaller the degree of the restriction of $\pi$ to $D_0$, forcing all components of $D_0$ to be vertical in extreme cases. 

In fact, a thorough investigation of the degree of $T$ (which may itself be nonreduced) as well as all multiplicities might yield another proof of Corollary \ref{alwayshorizontal} under the assumption that $Q$ not lie on six $(-1)$-curves of $S$, but it is not clear that this will require less computational effort than the given proof, especially given that even in the generic case, the intersection $S \cap T$ does not appear to admit a very elegant description (cf.~Remark \ref{whatisthiscurve}).
\end{remark}

\section{Torsion in a base change}\label{S:torsionbachange}

In this section, $k$ is still a field of characteristic not equal to $2$ or $3$.

\begin{lemma}\label{changefibers}
Let $B$ be a smooth curve over $k$ and $\pi \colon \E \to B$ a minimal nonsingular elliptic fibration. 
Let $C$ be a smooth curve over $k$ and $\tau \colon C \to B$ a nonconstant morphism. 
Let $\pi'\colon \E' \to C$ be the minimal nonsingular model of the base change $\E \times_B C \to C$ of $\pi$ by $\tau$. 
Let $c \in C(\kbar)$ be a point and set $b = \tau(c)$.
Let $\E_b$ and $\E_c'$ be the fibers of $\pi$ and $\pi'$ over $b$ and $c$, respectively. 
Let $e=e_c(\tau)$ be the ramification index of $\tau$ at $c$. Then the following statements hold. 
\begin{enumerate}
\item If $\E_b$ has type $I_d$ for some integer $d$, then $\E'_c$ has type $I_{de}$.
\item If $\E_b$ has type $I_d^*$ for some integer $d$, then $\E'_c$ has type $I_{de}$ for even $e$ and type $I_{de}^*$ for odd~$e$.
\item If $\E_b$ has type $IV^*$, then $\E'_c$ has type $I_0, IV^*, IV$ for $e \equiv 0,1,2 \pmod 3$, respectively.
\item If $\E_b$ has type $II$, then $\E'_c$ has type $I_0, II, IV, I_0^*, IV^*, II^*$ for $e \equiv 0,1,2,3,4,5 \pmod 6$, respectively.
\item If $\E_b$ has type $III$, then $\E'_c$ has type $I_0, III, I_0^*, III^*$ for $e \equiv 0,1,2,3 \pmod 4$, respectively.
\end{enumerate} 
\end{lemma}
\begin{proof}
This follows directly from Tate's algorithm (see \cite{tatealg} and \cite[IV.9.4]{silvermanadvanced}).
See also \cite[Table VI.4.1]{Miranda}, which is stated for characteristic zero. 
\end{proof}

\begin{lemma}\label{infoMN}
Suppose $k$ is algebraically closed.
Let $S$ be a del Pezzo surface of degree $1$ over $k$ and $\pi\colon \E \to \P^1$ the associated 
elliptic fibration. Let $M$ and $N$ denote the number of singular fibers of $\pi$ of type $I_1$ and type $II$, 
respectively. Then we have $M+2N=12$. If $\pi$ is not isotrivial, then we have $M \geq 4$.
\end{lemma}
\begin{proof}
Let $f,g \in k[z,w]$ be such that $S$ is given by 
\eqref{maineq}.  The surface $\E$ is a minimal nonsingular elliptic surface with fibers 
that are either nodal (type $I_1$) or cuspidal (type $II$). 
The discriminant $\Delta = 4f^3+27g^2$ vanishes at points of $\P^1$ corresponding to nodal and cuspidal fibers to order $1$ and $2$, respectively, so we get $M+2N=\deg \Delta = 12$. For any $t \in \P^1(\kbar)$ for which $\pi^{-1}(t)$ has type II, both $\Delta$ and the $j$-invariant $j = 2^83^3f^3/\Delta$ vanish at $t$ 
(see \cite[Tate's Algorithm, IV.9.4]{silvermanadvanced}), which implies that $f$ vanishes at $t$. It follows that 
$f$ vanishes at at least $N$ points, so if 
$M<4$, i.e, $N\geq 5$, then $f=0$, so $\pi$ is isotrivial. 
\end{proof}

\begin{remark}
In characteristic zero, the identity $M+2N=12$ follows from the more general fact that the Euler number of $\E$, which is $12$, equals the sum of the local Euler numbers, which are $1$ and $2$ for fibers of type $I_1$ and $II$, respectively (see \cite[Table IV.3.1]{Miranda} and \cite[Lemma IV.3.3]{Miranda}). 
The inequality $M< 4$ implies $N\geq 5$, which implies $N=6$ by \cite[Lemma 1.2]{persson}, which in turn implies that $\pi$ is isotrivial. 
\end{remark}

For $n\geq 3$, let $E(n) \to Y_1(n)$ be the universal elliptic curve over the usual modular curve $Y_1(n)$ over $\Z[1/n]$ with a section $P$ that has order $n$ in every fiber.
Then every elliptic curve $E$ over a scheme $S$ over $\Z[1/n]$ --with nowhere vanishing $j$-invariant if $n=3$-- with a section that has order $n$ 
in every fiber, is the base change of $E(n)/Y_1(n)$ by a unique morphism $S \to Y_1(n)$. 
Let $X_1(n)$ be the usual projective 
closure of $Y_1(n)$,  and let $\upsilon(n)\colon \mathbb{E}(n)\to X_1(n)$ 
be the minimal nonsingular elliptic fibration over $X_1(n)$ associated to $E(n)/Y_1(n)$. 
From Ogg's description of the cusps of $X_1(n)$ in \cite{oggtwo}, we conclude that for each $n\geq 5$ and each divisor $d$ of $n$, the number of fibers of $\upsilon(n)$ of type $I_d$ is $\frac{1}{2}\varphi(d)\varphi(n/d)$ (see also  \cite[p.\ 219]{Kub}, or \cite[Table 3]{Kub} for explicit models for small~$n$, which also show the types of the singular fibers). Table \ref{singfibsX1} 
gives the genus $g(X_1(n))$ of $X_1(n)$ (see \cite[p.\ 109]{oggone}) and 
describes the singular fibers of $\upsilon(n)$ for several $n$ (see 
\cite[Proposition 4.2]{shiodamodular}).
\begin{center}
\begin{table}[h]
\begin{tabular}{|r|c|l|}
\hline
$n$ & $g(X_1(n))$ & sing. fibers of $\upsilon(n)$ \\
\hline
3 & 0 & $IV^* + I_3 + I_1$  \\
5 & 0 & $2I_5+2I_1$ \\
7 & 0 &  $3I_7 + 3I_1$ \\
11 & 1 &  $5I_{11} + 5I_1$ \\
\hline
\end{tabular}
\caption{Singular fibers of $\upsilon(n)$}
\label{singfibsX1} 
\end{table}
\end{center}
To parametrize elliptic curves over a field of characteristic $p$ with a point of order $p$, we use Igusa curves instead of the modular curves above. For an extensive treatise of the subject, we refer the reader to \cite{igusa} and \cite[Chapter 12]{katzmazur}. For any prime $p\geq 3$, the smooth affine Igusa curve $\Ig(p)^{\ord}$ over $\F_p$ parametrizes ordinary elliptic curves $E$ with a point that generates the kernel of the Verschiebung map in the following sense (see \cite[Section 12.3 and Corollary 12.6.3]{katzmazur}). For every scheme $S$ over $\F_p$,
the absolute Frobenius $S \to S$ is the map that corresponds on affine rings to the map $x \to x^p$. 
For every elliptic curve $E \to S$, we let $E^{(p)}\to S$ denote the base change of $E \to S$ by the absolute Frobenius $S \to S$. 
By the universal property of the fibered product, the absolute Frobenius $E\to E$ factors as the composition of the projection $E^{(p)} \to E$ and a map $F=F_{E/S} \colon E \to E^{(p)}$ that we call the {\em relative Frobenius}. The dual isogeny $V=V_{E/S} \colon E^{(p)} \to E$ of $F_{E/S}$ is called the {\em Verschiebung}. There 
exists an elliptic curve $\mathfrak{E}(p)^\circ$ over the Igusa curve $\Ig(p)^{\ord}$, as well a section $\mathfrak{P}$ of the associated elliptic curve  ${\mathfrak{E}(p)^\circ}^{(p)} \to \Ig(p)^{\ord}$, such that all fibers of both fibrations are ordinary and $\mathfrak{P}$ generates the kernel of the Verschiebung $V \colon {\mathfrak{E}(p)^\circ}^{(p)} \to \mathfrak{E}(p)^\circ$, and such that for every elliptic curve $E$ over a scheme $S$ over $\F_p$ of which all fibers are ordinary, with a section $P$ of the associated curve $E^{(p)}\to S$ that generates the kernel of the Verschiebung $V \colon E^{(p)} \to E$, there is a unique morphism $\alpha \colon S \to \Ig(p)^{\ord}$ such that $E, E^{(p)}$, and $P$ are the base change of 
$\mathfrak{E}^\circ$, ${\mathfrak{E}(p)^\circ}^{(p)}$, and $\mathfrak{P}$, respectively, by $\alpha$.

If $k$ is a field of characteristic $p$ and $E'$ is an elliptic curve over $S=\Spec k$ with a point $P$ of order $p$, then there is an elliptic curve $E\to S$ such that $E^{(p)} \to S$ is isomorphic to $E'\to S$ and $P$ generates the kernel of Verschiebung; hence $E'\to S$ is a base change of the universal curve ${\mathfrak{E}(p)^\circ}^{(p)} \to \Ig(p)^{\ord}$. 

Let $\overline{\Ig(p)^{\ord}}$ denote the nonsingular projective completion of $\Ig(p)^{\ord}$, and let $\omega(p) \colon \mathfrak{E}(p)^{(p)}\to \overline{\Ig(p)^{\ord}}$ denote the minimal nonsingular projective model of ${\mathfrak{E}(p)^\circ}^{(p)} \to \Ig(p)^{\ord}$.

Table \ref{singfibsIg} gives the genus $g\big(\overline{\Ig(p)^{\ord}}\big)$ of 
$\overline{\Ig(p)^{\ord}}$ (see \cite[p.\ 96 and 99]{igusa}) and the fiber types of the singular fibers of $\omega(p)$ for several primes $p$. The fibers at the $(p-1)/2$ cusps have type $I_p$ \cite[Theorem 10.3]{liedtkeone} and the type of the fibers above the supersingular points can be deduced from \cite[Theorem 10.1]{liedtkeone}; for $p=13$ it suffices to note that the only supersingular $j$-value modulo $13$ is $5$, while for $p \in \{5,7,11\}$, the fibers are also given in \cite[Proposition 1.3]{liedtketwo}. This will be used in the proof of Theorem \ref{ellfibnotors}.

\begin{center}
\begin{table}[h]
\begin{tabular}{|r|c|l|}
\hline
$p$ & $g\big(\overline{\Ig(p)^{\ord}}\big)$ & sing. fibers of $\omega(p)$ \\
\hline
5 & 0 & $2I_5+II$ \\
7 & 0 &  $3I_7 + III$ \\
11 & 0 &  $5I_{11} + II+III$ \\
13 & 1 &  $6I_{13}+I_0^*$ \\
\hline
\end{tabular} 
\caption{Singular fibers of $\omega(p)$}
\label{singfibsIg} 
\end{table}
\end{center}
\begin{theorem}\label{ellfibnotors} 
Let $S$ be a del Pezzo surface of degree $1$ over $k$ and $\pi\colon \E \to \P^1$ the associated 
elliptic fibration. Let $C$ be a smooth, connected curve over $k$ of genus at most $1$, 
and $\tau \colon C \to \P^1$ a nonconstant morphism. Then the base change $\E \times_{\P^1} C \to C$ 
of $\pi$ by $\tau$ has no nonzero section of finite order. 
\end{theorem}
\begin{proof}   
Without loss of generality, we assume that $C$ is projective and that $k$ is algebraically closed.
As the curve $C$ is smooth and connected, it is integral, so it has a unique generic point $\eta$ that is dense in $C$. 
The curve $\E \times_{\P^1} \eta$ is an elliptic curve over the function field $\kappa(C)$ of $C$, which is an extension of the function field $k(t)$ of $\P^1$ with $t=z/w$. Let $j \in k(t)$ be the $j$-invariant of the generic fiber 
of $\pi$. Assume that the elliptic fibration $\E \times_{\P^1} C \to C$ has a nonzero section of finite order, say order $n>1$. 
Then the curve $\E \times_{\P^1} \eta$ has a point of order $n$ over $\kappa(C)$.
Without loss of generality, we assume that $n$ is prime.
Let $f,g \in k[z,w]$ be homogeneous polynomials such that $S$ is isomorphic to the surface in $\P$ given by \eqref{maineq}.
Let $M$ and $N$ denote the number of fibers of $\pi$ of type $I_1$ (nodal) and $II$ (cuspidal), respectively. 
Then $M+2N =12$ by Lemma \ref{infoMN}.
We will show that the genus of $C$ is at least $2$ by considering several cases, thus deriving the contradiction that proves the statement.

I) We first consider the case $n=2$. 
Note that $\pi$ itself has no section of order $2$, 
for if it did, it would be given by $y=0$ and $x=h(z,w)$ for some homogeneous polynomial $h \in k[z,w]$ of degree $2$, and then $S$ would be singular at the point on this section in the fibers given by $3h^2+f=0$. 
Since the locus $L\subset S$ of the $2$-torsion points has degree $3$ over $\P^1(z,w)$, it follows that $L$ is irreducible. To compute its genus, note that the map $\lambda \colon L \to \P^1$ ramifies whenever $D = 4f^3+27g^2$ vanishes. Moreover, if $D$ vanishes to order $1$ at $t \in \P^1$, which happens if and only if the fiber $\E_t$ of $\pi$ has type $I_1$, then 
there are two points on $L$ above $t$ with ramification indices $1$ and $2$, while if $D$ vanishes to order $2$, which happens if and only if the fiber $\E_t$ has type $II$, then there is a unique point on $L$ above $t$ with ramification index $3$. It follows that the degree of the ramification divisor of $\lambda$ equals $M+2N=12$, so the Riemann-Hurwitz formula applied to $\lambda$ shows that the genus of $L$ equals $1+ \frac{1}{2} (-2(\deg \lambda) + 12) = 4$. If the base change of $\pi$ by $\tau$ has a section of order $2$, then this section would map nontrivially to $L$, so we get $g(C) \geq 4$. 

II) We now consider the case that $j$ is constant, that is $j \in k$, and may assume $n\neq 2$. 
Then there are $a,b \in k$ and $h \in \overline{k(t)}$, where $\overline{k(t)}$ denotes an algebraic closure of $k(t)$, such that $f(t,1) = ah^2$ and 
$g(t,1)=bh^3$. If $f,g \neq 0$, then $h  = ab^{-1}g(t,1)f(t,1)^{-1}$ is contained in $k(t)$ and one checks that $S$ is not smooth. If $g=0$, then $S$ is not smooth either, so we find $f=0$ and again from smoothness of $S$, we find that $g(t,1)$ is separable and has degree $5$ or $6$ in $t$ (cf.~\cite[Proposition 3.1]{varilly}). 
Suppose $P = (x_1,y_1) \in \E \times_{\P^1} \eta$ is a point over $\kappa(C)$ of order $n$. Let $\overline{\kappa(C)}$ be an algebraic closure of $\kappa(C)$ and
$\beta \in \overline{\kappa(C)}$ an element satisfying $\beta^6 = g(t,1)$. Then $(x_1\beta^{-2},y_1\beta^{-3})$ is a point of order $n$ on the curve given by $y^2 = x^3+1$, so there are $x_2,y_2 \in k$ such that $x_1 = x_2\beta^2$ and $y_1 = y_2\beta^3$. From $n \neq 2$, we get $y_1 \neq 0$. If $x_1 \neq 0$, then $\beta = x_2y_2^{-1}y_1x_1^{-1}$ is contained in $\kappa(C)$; 
if $x_1=0$, then $y_1^2 = g(t,1)$, so in any case, $g(t,1)$ is a square in $\kappa(C)$, which implies that $\kappa(C)$ contains a subfield of genus $2$, so $C$ has at least genus $2$ itself. 

III) The case $n \geq 3$ and $j\not \in k$. If the characteristic of $k$ is not equal to $n$, then we set $Y = Y_1(n)$ and $X = X_1(n)$ and  $\mathbb{E} = \mathbb{E}(n)$ and $\upsilon = \upsilon(n)$; otherwise, we set $Y = \Ig(n)^\ord$ and $X = \overline{\Ig(n)^{\ord}}$ and $\mathbb{E} =\mathfrak{E}(n)^{(n)}$ and $\upsilon = \omega(n)$. 
In either case, there is a morphism $\eta \to Y \subset X$ such that the elliptic curve $\E \times_{\P^1} \eta$
over $\eta$ is the base change of $\mathbb{E}$ over $X$. This 
morphism extends to a morphism $\chi \colon C \to X$, which is nonconstant because $j$ is not constant.
The elliptic surfaces $\E \times_{\P^1} C$ and $\mathbb{E} \times_X C$ have isomorphic generic fibers $\E \times_{\P^1} \eta \isom \mathbb{E} \times_{X} \eta$, so their minimal nonsingular models are isomorphic as well by 
\cite[Proposition II.1.2 and Corollary II.1.3]{Miranda}. 
Let $\pi'\colon \E' \to C$ be this minimal nonsingular elliptic fibration.
$$
\xymatrix{
\E \ar[d]_{\pi} & \E \times_{\P^1} C \ar[l] \ar@{<-->}[r] \ar[d]& \E' \ar@{<-->}[r] \ar[dr]_{\pi'} \ar[dl]^{\pi'}
      & \mathbb{E} \times_{X} C \ar[r]\ar[d] & \mathbb{E} \ar[d]^{\upsilon}\\
\P^1 & C \ar[l]^{\tau} \ar@{=}[rr] & & C \ar[r]_{\chi} & X
}
$$
Set $d = \deg \tau$ and let $R \in \Div C$ denote the ramification divisor of $\tau$. Then the degree of $R$ is at least
\begin{equation}\label{degR}
\sum_{c \in C} (e_{c}(\tau)-1) \geq  
\sum_{\substack{b \in\P^1\\ \E_b \mbox{\tiny\ type } I_1}}\left(\sum_{\substack{c\in C\\ \tau(c)=b}}(e_{c}(\tau)-1)\right)+
\sum_{\substack{b \in\P^1\\ \E_b \mbox{\tiny\ type } II}}\left( \sum_{\substack{c\in C\\ \tau(c)=b}}(e_{c}(\tau)-1)\right), 
\end{equation}
where $e_{c}(\tau)$ denotes the ramification index of $\tau$ at $c$. 

Lemma \ref{changefibers} relates the types of the singular fibers of $\pi'$ to those of $\pi$ and the ramification of $\tau$ on one hand, and to those of $\upsilon$ and the ramification of $\chi$ on the other hand. The remainder of the proof consists of a largely combinatorial argument to give a lower bound for the degree of $R$, which then, by the Riemann-Hurwitz formula, yields a lower bound for the genus of $C$. 

Lemma \ref{changefibers} implies that the points $c \in C$ for which the fiber $\E_{\tau(c)}$ of $\pi$ above $\tau(c)$ has type $I_1$ are exactly the points for which the fiber $\E'_c$ of $\pi'$ above $c$ has type $I_m$ for some integer $m\geq 1$, and exactly the points for which the fiber $\mathbb{E}_{\chi(c)}$ of $\upsilon$ above $\chi(c)$ has type $I_j$ or $I_j^*$ for some integer $j\geq 1$; for such points $c$, and integers $m$ and $j$, the quotient ${\ell}=m/j$ is a positive integer and we have $e_c(\tau) = j{\ell}$ and $e_c(\chi) = {\ell}$. 
For each $j \geq 1$, let $r_j$ denote the number of fibers of $\upsilon$ of type $I_j$ or $I_j^*$; for each ${\ell}\geq 1$, let $s_{j,{\ell}}$ denote the number of fibers of $\pi'$ of type $I_{j{\ell}}$ that lie above a point $c \in C$ for which the fiber of $\upsilon$ above $\chi(c)$ has type $I_j$ or $I_j^*$. 
For every $x \in X$ we have $\sum_{c\in \chi^{-1}(x)} e_c(\chi) = \deg \chi$. Summing over all $x\in X$ for which the fiber $\mathbb{E}_x$ has type $I_j$, we find $\sum_{\ell\geq 1} \ell s_{j,\ell} = (\deg \chi) \cdot r_j$ for all $j\geq 1$. 
The same argument applied to $\tau$ yields 
$$
Md = \sum_{j,{\ell}\geq 1} j{\ell}s_{j,{\ell}} = (\deg \chi)\cdot \sum_{j\geq 1} jr_j.
$$

It follows that the first term of the right-hand side of \eqref{degR} equals 
\begin{equation}\label{fractiontobebounded}
\sum_{j,{\ell} \geq 1} (j{\ell}-1)s_{j,{\ell}} \geq \sum_{j,{\ell}\geq 1} (j-1){\ell}s_{j,{\ell}} = (\deg \chi)\cdot \sum_{j\geq 1} (j-1)r_j = 
                              \frac{\sum_{j\geq 1} (j-1)r_j}{\sum_{j\geq 1} jr_j}\cdot Md.
\end{equation}
We consider two subcases. 

A) The characteristic of $k$ is not equal to $n$. 
From $g(X) \leq g(C) \leq 1$ we conclude $n\leq 12$ or $n=14$ or $n=15$ (see \cite[p.\ 109]{oggone}), and since $n\geq 3$ is prime, we have $n \in \{3,5,7,11\}$. From Table \ref{singfibsX1} above, we find that the fraction in the right-most expression of \eqref{fractiontobebounded} is at least $\frac{1}{2}$.
Since $\upsilon = \upsilon(n)$ has only fibers of type $IV^*$, $I_1$, and $I_n$, Lemma \ref{changefibers} implies that $\pi'$ does not have fibers of type $II$, $II^*$, or $I_0^*$. Again from Lemma \ref{changefibers}, this time viewing $\pi'$ as the minimal model of the base change of $\pi$ by $\tau$, we find that for every $c \in C$ for which the fiber of $\pi$ above $\tau(c)$ has type $II$, the ramification index $e_{c}(\tau)$ is even, so we have $e_{c}(\tau)-1 \geq \frac{1}{2}e_{c}(\tau)$. Therefore, the second term of the right-hand side of \eqref{degR} is at least $\frac{1}{2}Nd$, so the degree of $R$ is at least $\frac{1}{2}Md + \frac{1}{2}Nd \geq \frac{1}{4}d(M + 2N) = 3d$. The Riemann-Hurwitz formula applied to $\tau$ then yields $2g-2 = -2d + \deg R \geq d >0$, so $g>1$. 

B) The characteristic of $k$ is equal to $n$. From $g(X)\leq g(C) \leq 1$ we conclude $n \in \{5,7,11,13\}$ (for a formula for the genus of $X$, see \cite[p.\ 96 and 99]{igusa}).  
From Table \ref{singfibsX1} above, we find that the fraction in the right-most expression of \eqref{fractiontobebounded} is at least $\frac{4}{5}$. Also, from the fact that $\pi$ is not isotrivial, we get $M\geq 4$ by Lemma \ref{infoMN}, so the degree of $R$ is at least $\frac{4}{5}Md >3d$. As before, the Riemann-Hurwitz formula yields $g > 1$. 
\end{proof}

\section{Proof of the main theorems}\label{proofs}

In this section, the field $k$ is still of characteristic different from $2$ and $3$. 

\begin{theorem}\label{main}
Suppose $k$ is infinite. 
Let $S\subset \P$ be a del Pezzo surface of degree $1$ over $k$, given by \eqref{maineq} for some homogeneous $f,g \in k[z,w]$ of 
degree $4$ and $6$, respectively. Let $Q = (x_0:y_0:0:1) \in S(k)$ be a rational point with $y_0 \neq 0$. 
Suppose that the following statements hold.
\begin{itemize}
\item If the order of $Q$ in $S_0^\ns(k)$ is at least $4$, then $\C_Q(5)$ has infinitely many $k$-points. 
\item If the characteristic of $k$ equals $5$, then the order of $Q$ in $S_0^\ns(k)$ is not $5$.
\item The pair $(S,Q)$ is not isomorphic to a pair described in Examples \ref{ex:hell5Q}, \ref{ex:order3}(iii), or \ref{ex:order3sec}(iii).
\item If the pair $(S,Q)$ is isomorphic to a pair described in Subexamples \ref{ex:order3}(i) or \ref{ex:order3sec}(i), then the set of $k$-points on $\C_Q(5)$ is Zariski dense in $\C_Q(5)$.
\end{itemize}
Then the set $S(k)$ of $k$-points on $S$ is Zariski dense in $S$.  
\end{theorem}
\begin{proof}
Given $S$ and $Q$, we let the curve $\C_Q(5)$, its completion $\pC_Q(5)$, the rational map $\sigma \colon \pC_Q(5) \dashrightarrow S$, the elliptic fibration $\pi\colon \E \to \P^1$, and the element $\phi_3\in k$ be as in Sections \ref{multisection} and \ref{S:completion}.
We claim that there exists an irreducible component $\C_0$ of $\pC_Q(5)$ for which $\sigma(\C_0)$ is a horizontal curve on $S$ and $\C_0(k)$ is infinite. Indeed, if the order of $Q$ in $S_0^\ns(k)$ is at least $4$, then $\phi_3\neq 0$, so $\C_Q(5)$ 
is a double cover of $\A^1(p)$, the curve $\C_Q(5)$ has at most two irreducible components, and if there are two, then there is an involution that switches them, so the first assumption of the theorem implies that $\C_Q(5)(k)$ is Zariski dense in $\C_Q(5)$; thus, there exists an irreducible component $\C_0$ of $\C_Q(5)$ that satisfies the claim by Corollary \ref{alwayshorizontal}.
Suppose, for the remainder of this paragraph and the proof of the claim, that the order of $Q$ is $3$. Then for any pair $(S,Q)$ that is not isomorphic to one of the pairs described in Examples \ref{ex:order3} and \ref{ex:order3sec}, the unique component of $\C_Q(5)$ that projects birationally to $\A^1(p)$ satisfies the claim by Proposition \ref{nohell3Q}. For any pair $(S,Q)$ that {\em is} isomorphic to one of the pairs described in those examples, 
the curve $\C_Q(5)$ contains a component $\C_0$ whose projection to $\A^1(p)$ is constant;
this component $\C_0$ satisfies the claim, as its image is horizontal by Proposition \ref{nohell3Q} (by assumption we are not in the Subexample \ref{ex:order3}(iii) or \ref{ex:order3sec}(iii)) and density of $\C_0(k)$ follows either automatically in the case of Subexample~(ii) or by assumption in the case of Subexample~(i). 

Let $\C_0$ be a component of $\pC_Q(5)$ as in the claim, and let $\tilde{\C}_0$ be a normalization of $\C_0$. Then the rational map $\sigma \colon \pC_Q(5) \dashrightarrow S$ induces a morphism $\tilde{\sigma} \colon \tilde{\C}_0 \to \E$. The composition $\pi \circ \tilde{\sigma} \colon \tilde{\C}_0 \to \P^1$ corresponds on an open subset to the rational map $\varphi \circ \sigma \colon \pC_Q(5) \dashrightarrow \P^1$, so it is surjective by the claim. 
Let $\theta$ denote the section $\id \times \tilde{\sigma} \colon \tilde{\C}_0 \to \tilde{\C}_0 \times_{\P^1} \E$ of the elliptic surface $ \tilde{\C}_0 \times_{\P^1} \E \to \tilde{\C}_0$. 
$$
\xymatrix{
\tilde{\C}_0 \times_{\P^1} \E \ar[dd] \ar[rr] && \E \ar[dd]^{\pi} \ar[dr] \\
&&&S \ar@{-->}[dl]^{\varphi} \\
\tilde{\C}_0 \ar[rr]_{\pi \circ \tilde{\sigma}} 
\ar@/^5.5mm/[uu]^{\theta}
\ar[rruu]^{\tilde{\sigma}} && \P^1
}
$$
The section $\theta$ is not the zero section because $\sigma \colon \pC_Q(5) \dashrightarrow S$ sends points $P \in \C_Q(5)$ whose associated curve $C$ is not contained in $S$ to a point unequal to $\O \in S$. 
The genus of $\tilde{\C}_0$ is at most $1$, so by Theorem \ref{ellfibnotors}, the section $\theta$ has infinite order. 
Since $\tilde{\C}_0(k)$ is Zariski dense in $\tilde{\C}_0$, it follows that the rational points are dense on the images of all infinitely many multiples of $\theta$. 
Thus, the $k$-rational points are dense in the surface $\tilde{\C}_0 \times_{\P^1} \E$ and as this surface maps dominantly to $S$, we conclude that $S(k)$ is Zariski dense in $S$. 
\end{proof}

Obviously, for any point $Q' \in S(k) - \{\O\}$, we may apply an automorphism of $\P^1(z,w)$ to ensure that we have $\varphi(Q') = (0:1)$, so the implicit assumption in Theorem \ref{main} and related statements that $\varphi(Q)=(0:1)$ is not a restriction.    

Note that if the order of $Q$ in $S_0^\ns(k)$ is $3$ and the pair $(S,Q)$ is not isomorphic to a pair described in 
Examples \ref{ex:order3}(i) and \ref{ex:order3sec}(i), then the hypotheses of Theorem \ref{main} are automatically satisfied, without further assumptions on $\C_Q(5)$. 

\begin{example}\label{smallsample}
We took a small sample of approximately a hundred randomly chosen del Pezzo surfaces over $\Q$ given by \eqref{maineq} with $f$ and $g$ having only coefficients $0$, $1$, and $-1$. For nearly half of the cases, a short point search revealed a rational point $Q$ for which we could show that it satisfies all conditions of Theorem \ref{main}, thus proving the rational points are Zariski dense. For the remaining cases, we could still find points $Q$, but
the coefficients of $\C_Q(5)$ were too large to show that $\C_Q(5)$ has infinitely many rational points.  
\end{example}

\begin{example}\label{ex:tony}
As mentioned in Section \ref{intro}, A.~V\'arilly-Alvarado proves in \cite[Theorem 2.1]{varilly}
that if we have $k=\Q$ and $f=0$, and some technical conditions on $g$, as well as a finiteness conjecture hold, 
then the set of rational points is Zariski dense on the surface given by \eqref{maineq}. 
He also mentions the surface $S$ with $f=0$ and $g=243z^6+16w^6$ as an example that would not succumb to his 
methods, so we took $S$ as a test example for our method.
Unfortunately, the point $(0:4:0:1)$ of order $3$ on $S_0\subset S$ lies on nine $(-1)$-curves (cf.~Example \ref{ex:order3sec}(iii)). It is not hard to find more rational points on this surface, but we did not succeeded in finding any points on the curve $\C_Q(5)$ associated to any of these points $Q$ as the coefficients are rather large: for the second-smallest point $Q = (-63:14:1:5)$, the conductor of the Jacobian of $\C_Q(5)$ has $62$ digits. N.~Elkies did prove that the points on $S$ are dense with a different method \cite{elkies}. 
\end{example}

\begin{proof}[Proof of Theorem \ref{maincor}]
The fact that $Q$ is not fixed by the automorphism that changes the sign of $y$ implies $Q \neq \O$. 
Without loss of generality, we assume $\varphi(Q) = (0:1)$, say $Q = (x_0:y_0:0:1)$, with $y_0 \neq 0$. 
Hence, we may apply Theorem \ref{main}.  
The last hypothesis of Theorem \ref{maincor} implies the last two of Theorem \ref{main}, which shows that $S(k)$
is indeed Zariski dense in $S$. 
\end{proof}

\begin{proof}[Proof of Theorem \ref{denseinmodulispace}]
Note that any point $(x_0,y_0)$ on an elliptic curve given by $y^2 = x^3+ax+b$ has order $3$ if and only if 
$(a+3x_0^2)^2 = 12x_0y_0^2$.
Define the polynomials $f = \sum_{i=0}^4 f_iu^i$ and $g = \sum_{j=0}^6 g_ju^j$. 
Suppose we have $\ell \in \{0,\ldots,4\}$, $m \in \{0,\ldots, 6\}$, and $\varepsilon >0$. 
Since every elliptic curve over the real numbers $\R$ has a nontrivial $3$-torsion point, we may choose a nonzero rational number $t \in \Q^*$ and a point $Q = (x_0:y_0:t:1) \in S(\R)$ such that 
the fiber $S_t$ given by $y^2 = x^3 + f(t)x+g(t)$ is smooth, the point $Q$ has order $3$ in $S_t(\R)$, and $Q$ does not lie on six $(-1)$-curves on $S$. 
Set $\xi_0 = \frac{1}{6}y_0^{-1}(f(t)+3x_0^2)$, so that $Q$ being $3$-torsion implies $3\xi_0^2 = x_0$. 
Choose $\xi_1,y_1 \in \Q^*$ close to $\xi_0$ and $y_0$, respectively, and set $x_1 = 3\xi_1^2$ and $Q' = (x_1:y_1:t:1)$. 
Also set 
\begin{align*}
\lambda &= f_\ell + t^{-\ell}(6\xi_1 y_1-3x_1^2-f(t)), \\
\mu &= g_m + t^{-m}(y_1^2-x_1^3-(6\xi_1 y_1-3x_1^2)x_1-g(t)),\\
f'  &= f - f_\ell u^\ell + \lambda u^\ell,\\
g'  &= g - g_m u^m + \mu u^m,
\end{align*}
so that $f'$ and $g'$ are the polynomials obtained from $f$ and $g$ after replacing $f_\ell$ and $g_m$ by $\lambda$ and $\mu$, respectively. Then we have 
$f'(t) = 6\xi_1 y_1-3x_1^2$ and $g'(t) = y_1^2-x_1^3-f'(t)x_1$, so $Q$ lies on the surface $S'$ given by \eqref{tobechanged} with the two values $f_\ell$ and $g_m$ replaced by $\lambda$ and $\mu$, respectively. If we choose $\xi_1$ and $y_1$ arbitrarily close to $\xi_0$ and $y_0$, then $\lambda$ and $\mu$ will be arbitrarily close to $f_\ell$ and $g_m$. By choosing them close enough, we also guarantee that $S'$ and $S_t'$ are smooth, and that $Q'$ does not lie on six $(-1)$-curves on $S'$. From the identity 
$(f'(t) + 3x_1^2)^2 = 36\xi^2 y_1^2 = 12x_1y_1^2$ we conclude that $Q'$ has order $3$ in $S'_t(\Q)$, so we may apply 
Theorem \ref{maincor}, which yields that $S'(\Q)$ is Zariski dense in $S'$.
\end{proof}

\begin{lemma}\label{specinforder}
Let $k$ be an infinite field and $X \to \P^1$ an elliptic fibration over $k$ with a nontorsion section. Then there are infinitely many points $t \in \P^1(k)$ for which the fiber $X_t$ contains infinitely many $k$-rational points.
\end{lemma}
\begin{proof}
If $k$ is algebraic over a finite field, then this follows from the Weil conjectures. Otherwise, we replace $k$ without loss of generality by an infinite subfield that is finitely generated over its prime subfield, over which everything is defined. Then $k$ is either a number field or a transcendental extension of its prime field and in all cases, $k$ is Hilbertian (see \cite{hilbert} for number fields and \cite[Theorem 13.4.2]{friedjarden} for a modern treatment of the general case). 
The lemma now follows immediately from N\'eron's Specialization Theorem \cite[Th\'eor\`eme IV.6]{neron} (see also \cite[Theorem 7.2]{lang} and \cite[Remark 3.7(1)]{petersen}).
\end{proof}

\begin{proof}[Proof of Theorem \ref{densewhennodal}]
Without loss of generality, we assume that the nodal fiber lies above $(0:1)$. When $Q$ runs over the nodal curve $S_0$, the curves $\pC_Q(5)$ form a family of genus-one curves. More precisely, the equation in \eqref{completion} describes a surface $X \subset \A^1(x_0) \times \P(1,2,1)$, and if $Q = (x_0:y_0:0:1)$ is a point on $S_0$ with $y_0 \neq 0$ and not of order $3$ in $S_0^\ns(k)$, then the fiber of the projection $\mu \colon X \to \A^1$ above $x_0$ is isomorphic to $\pC_Q(5)$. The fibered product $(S_0-\{\O\}) \times_{\A^1} X$ is the family of curves $\pC_Q(5)$, at least outside finitely many points $Q \in S_0$. 
Let $d \in k^*$ be such that $f_0 = -3d^2$ and $g_0 = 2d^3$. 
Note that we have two rational maps $\chi_i \colon \A^1 \dashrightarrow X$, for $i=1,2$, that are rational 
sections of $\mu$, namely given by $x_0 \mapsto (x_0, (1:\alpha_i:0))$ with 
$\alpha_1=\frac{1}{4}(x_0+2d)^{-1}$ and $\alpha_2 = \frac{1}{4}(x_0+7d)(x_0+2d)^{-1}(x_0+3d)^{-1}$ as in Lemma \ref{L:E0node}.
These maps extend to morphisms and we choose $\chi_1$ to be the zero section, making $\mu$ a Jacobian elliptic fibration.

We claim that the section $\chi_2$ has infinite order. 
The model $X \to \A^1$ is highly singular, so instead we consider the surface $X' \subset \A^1(x_0) \times \P(1,2,1)$ that is the image of $X$ under the birational map 
\begin{align*}
\A^1 \times \P(1,2,1) &\dashrightarrow \A^1 \times \P(1,2,1) \\
(x_0,(\ovp:\ovq:\ovr)) &\mapsto (x_0,(\ovp':\ovq':\ovr'))
\end{align*}
with 
\begin{align*}
\ovp' & = 8(x_0-d)^2 \ovp +(x_0-d)(f_1d+g_1)\ovr, \\
\ovq' & = 2\varphi_2^{-1}(2c_1\ovq + c_2\ovp^2 + c_3\ovp\ovr + c_4 \ovr^2),\\
\ovr' & = 8\ovr,
\end{align*}
where $\phi_2, c_1, c_2, c_3, c_4$ depend on $x_0$ (which is now variable instead of fixed) as they did before. 
Note that $f_1d+g_1$ is nonzero because $S$ is smooth. For $x_0 \not \in \{d,-2d,-3d\}$, the fibers of $\mu \colon X\to \A^1$ and $\mu'\colon X'\to \A^1$ are isomorphic. 
The model $X'$ is given by $\ovq^2 = H(\ovp,\ovr)$, where $H\in k[x_0][\ovp,\ovr]$ is homogeneous of degree $4$. 
The fiber $X'_d$ of $\mu'$ above $x_0 = d$ is given by 
\begin{equation}\label{fiberatd}
\ovq^2 = 81 d^4 (f_1d+g_1) \ovp^2\ovr(\ovp + (f_1d+g_1)\ovr).
\end{equation}
This fiber $X'_d$ is singular at the point $(d,(0:0:1))$, and in fact so is $X'$, but the fiber is smooth everywhere else. 
The sections $\chi_1$ and $\chi_2$ correspond to the sections $\chi_1' \colon x_0 \mapsto (x_0,(4: 6d(d-x_0) : 0 ))$ 
and $\chi_2' \colon x_0 \mapsto (x_0,(4: 6d(x_0-d) : 0 ))$ of $\mu'$, respectively.
These sections intersect in the point $(d,(1:0:0))$, which is smooth in its fiber. 
Therefore, in a minimal nonsingular projective model $\overline{\mu}\colon \overline{X}\to \P^1$ 
of the fibration $\mu$, the two sections intersect as well. 
Hence, $\chi'_2$ is in the kernel of the reduction $\overline{X}(\P^1) \to \overline{X}_d(k)$, where $\overline{X}_d$
is the fiber of $\overline{\mu}$ above $x_0=d$. 
This kernel is isomorphic to a subgroup of the formal group associated to $\overline{\mu}$ (or $\mu'$) 
and the completion of $k[x_0]$ at the maximal ideal $(x_0-d)$, cf.~\cite[Proposition VII.2.2]{silverman}. 
By \cite[Proposition IV.3.2(b)]{silverman}, all torsion elements of the formal group 
have $p$-power order, where $p$ is the characteristic of $k$. 
This proves the claim for $p=0$ (cf.~\cite[Theorem 1.1(a)]{MP}). 
We now assume $p>0$ and determine the Kodaira type of the singular 
fiber $\overline{X}_d$ of $\overline{\mu}$. One checks that the discriminant of $H$ equals  
$$
\Delta = (x_0-d)^3(x_0+2d)^8(x_0+3d)^2D(x_0),
$$
where $D$ is a polynomial of degree $35$ satisfying $2^{11}D(d) = -3^{13}d^{11}(f_1d+g_1)^{12}\neq 0$. 
Hence, the valuation of $\Delta$ at $x_0=d$ equals $3$; the fiber $X'_d$ described in \eqref{fiberatd} is nodal, so the reduction is multiplicative and we conclude from \cite[Tate's Algorithm IV.9.4]{silvermanadvanced} that $\overline{X}_d$ has type $I_3$.
It follows that the $j$-invariant of $\mu$ is not constant.
Suppose that $\chi'_2$ is torsion. Then $\overline{\mu}$ admits a section of order $p$, so there is a surjective morphism 
$\psi \colon \P^1 \to \overline{\Ig(p)^{\ord}}$ such that the generic fiber of $\overline{\mu}$ is isomorphic to the generic fiber of the base change of the fibration $.(p) \colon \mathfrak{E}(p)^{(p)} \to \overline{\Ig(p)^{\ord}}$ by $\psi$ (cf.~part III of the proof of Theorem \ref{ellfibnotors}).
This implies that the minimal nonsingular model of this base change is isomorphic to $\overline{\mu}$.
However, the existence of $\psi$ implies that $\overline{\Ig(p)^{\ord}}$ has genus $0$, so $p \leq 11$, and from Lemma \ref{changefibers} and Table \ref{singfibsIg} we find that no base change of $\omega(p)$ has a minimal nonsingular model with fibers of type $I_3$. This contradicts the fact that $\overline{X}_d$ has type $I_3$ and the claim follows.

It follows that the section $(S_0-\{\O\}) \dashrightarrow (S_0-\{\O\}) \times_{\A^1} X$ induced by $\chi_2$ also has 
infinite order on the elliptic fibration $(S_0-\{\O\}) \times_{\A^1} X \to S_0-\{\O\}$ with the section induced by $\chi_1$ as zero section. After replacing $S_0$ by its normalization we may apply Lemma \ref{specinforder}, which implies that the curve $\pC_Q(5)$ has infinitely many rational points for infinitely many $Q \in S_0^\ns(k)$, in particular for some $Q$ of order larger than $5$ in $S_0^\ns(k)$. Theorem \ref{maincor} then shows that $S(k)$ is Zariski dense in $S$.
\end{proof}

\end{document}